\documentclass[UTF8,10pt]{article}
	
	\usepackage[left=3.18cm,right=3.18cm,top=2.54cm,bottom=2.54cm]{geometry}
	\usepackage{amsfonts}
	\usepackage{amsmath}
	\usepackage{amssymb}
	\usepackage{pgfplots}
	\usepackage{tikz}
	\usetikzlibrary{arrows}
	\usepackage{titlesec}
	\usepackage{bm}
	\usepackage{appendix}
	\usepackage{tikz-cd}
	\usepackage{mathtools}
	\usepackage{amsthm}
	\usepackage{setspace}
	\usepackage{enumitem}
	\setenumerate[1]{itemsep=0pt,partopsep=0pt,parsep=\parskip,topsep=5pt}
	\setitemize[1]{itemsep=0pt,partopsep=0pt,parsep=\parskip,topsep=5pt}
	\setdescription{itemsep=0pt,partopsep=0pt,parsep=\parskip,topsep=5pt}
	\usepackage[colorlinks,linkcolor=black,anchorcolor=black,citecolor=black]{hyperref}
	\usepackage[numbers]{natbib}
	\usepackage{cases}
	\usepackage{fancyhdr}
	\usepackage[scaled=0.92]{helvet}	
	\usepackage{upgreek}
	\usepackage{txfonts}
	\usepackage{verbatim}

	\theoremstyle{definition}
	\newtheorem{thm}{Theorem}[section]

	\newtheorem{lem}[thm]{Lemma}
	\newtheorem{prop}[thm]{Proposition}
	
	\newtheorem{cor}[thm]{Corollary}
	
	\newtheorem{rmk}{Remark}[section]

	\newcommand{\R}{\mathbb{R}}  
	\newcommand{\Z}{\mathbb{Z}}
	
	\newcommand{\N}{\mathbb{N}}
	
	\newcommand{\T}{\mathbb{T}}

	\newcommand{\RR}{\mathcal{R}}      
	\newcommand{\Om}{\Omega}
	\newcommand{\q}{\quad}
	\newcommand{\p}{\partial}

	\newcommand{\nab}{\nabla}

	\newcommand{\lap}{\Delta}

	\newcommand{\cnab}{\overline{\nab}}
	
	\newcommand{\dx}{\,\mathrm{d}x}

	\newcommand{\dt}{\,\mathrm{d}t}
	
	\newcommand{\dtau}{\,\mathrm{d}\tau}

	\newcommand{\W}{\mathcal{W}}

	\newcommand{\wg}{\widetilde{\mathcal{G}}}

	\newcommand{\io}{\int_{\Omega}}

	\newcommand{\ddt}{\frac{\mathrm{d}}{\mathrm{d}t}}
	\numberwithin{equation}{section}
	
	\newcommand{\eps}{\varepsilon}
	\newcommand{\BF}{\bar{F}}
	\newcommand{\fjp}{(F_j+\bar{F}_j)\cdot\nab}
	\newcommand{\flp}{(F_l+\bar{F}_l)\cdot\nab}
	\newcommand{\fj}{(F_j+\bar{F}_j)}
	\newcommand{\fl}{(F_l+\bar{F}_l)}
	\usepackage{xcolor}

	\allowdisplaybreaks[1]

\begin{document}
		\title{\bf Low Mach Number Limit of Non-isentropic Inviscid Elastodynamics with General Initial Data}
		\date{\today}
		\author{{\sc Jiawei WANG}\thanks{School of Mathematics and Statistics, Fuzhou University, Fuzhou, Fujian, P.R. China.
				Email: \texttt{wangjiawei9609@163.com}}\,\,\,\,\,\, {\sc Junyan ZHANG}\thanks{School of Mathematical Sciences, University of Science and Technology of China. 96 Jinzhai Road, Baohe District, Hefei, Anhui, P.R. China. 
				Email: \texttt{yx3x@ustc.edu.cn}} }
		\maketitle
		
		\begin{abstract}
			We prove the incompressible limit of non-isentropic inviscid elastodynamic equations with {\it general initial data} in 3D half-space. The deformation tensor is assumed to satisfy the neo-Hookean linear elasticity and degenerates in the normal direction on the solid wall. The uniform estimates in Mach number are established based on two important observations. First, the entropy has enhanced regularity in the direction of each column of the deformation tensor, which exactly helps us avoid the loss of derivatives caused by the simultaneous appearance of elasticity and entropy in vorticity analysis. Second, a special structure of the wave equation of the pressure together with elliptic estimates helps us reduce the normal derivatives in the control of divergence and pressure. The strong convergence of solutions in time is obtained by proving local energy decay of the wave equation and using the technique of microlocal defect measure.
			
		\end{abstract}
		
		
		\setcounter{tocdepth}{1}
		\tableofcontents
		
		\section{Introduction}
		
		We consider 3D compressible inviscid elastodynamic equations in three spatial dimensions
		\begin{equation}
			\begin{cases}
				D_t\rho+\rho(\nabla\cdot u)=0~~~&\text{in}~[0,T]\times\Om, \\
				\rho D_t u +\eps^{-2} \nabla p= \nabla\cdot(\rho\dot{\boldsymbol{F}}\dot{\boldsymbol{F}}^{\mathrm{T}})~~~& \text{in}~[0,T]\times\Om, \\
				D_t \dot{\boldsymbol{F}}=\nabla u  \dot{\boldsymbol{F}}~~~&\text{in}~[0,T]\times\Om, \\
				\nabla\cdot (\rho \dot{\boldsymbol{F}})=0~~~&\text{in}~[0,T]\times\Om,\\
				D_t S=0~~~&\text{in}~[0,T]\times\Om,
			\end{cases}\label{CElasto}
		\end{equation}
		Here $\Om=\mathbb{R}^3_-:=\{x\in \mathbb{R}^3: x_3<0\}$ is the half-space with boundary $\Sigma:=\{x_3=0\}$. $\nabla:=(\p_{x_1},\p_{x_2},\p_{x_3})^{\mathrm{T}}$ is the standard spatial derivative. $N=(0,0,1)^{\mathrm{T}}$ is the unit outward normal of $\Sigma$. $D_t:=\p_t+u\cdot\nabla$ is the material derivative. The fluid velocity, the deformation tensor, the fluid density, the fluid pressure and the entropy are denoted by $u=(u_1,u_2,u_3)^{\mathrm{T}}$, $\dot{\boldsymbol{F}}=(\dot{F}_{ij})_{3\times3}$, $\rho$, $p$ and $S$ respectively. $\rho\dot{\boldsymbol{F}}\dot{\boldsymbol{F}}^{\mathrm{T}}$ is the Cauchy-Green stress tensor in the case of compressible neo-Hookean linear elasticity. The fourth equation of \eqref{CElasto} will not make the system be over-determined because we only require it holds for the initial data and it automatically propagates to any time (cf. Trakhinin \cite[Prop. 2.1]{Trakhinin2016elastic}). Note that the last equation of \eqref{CElasto} is derived from the equation of total energy and Gibbs relation. The Mach number $\eps$, defined as the ratio of characteristic fluid velocity to the sound speed, is a dimensionless parameter that measures the compressibility of the fluid. The inviscid elastodynamic system describes the motion of a neo-Hookean elastic medium corresponding to the elastic energy $W(\dot{\boldsymbol{F}})=\frac12|\dot{\boldsymbol{F}}|^2$. It also arises as the inviscid limit of the compressible visco-elastodynamics of the Oldroyd type \cite{Dafermos10}. 
		
		We assume the fluid density $\rho=\rho(p,S)>0$ to be a given smooth function of $p$ and $S$ which satisfies 
		\begin{align}
			\rho\geq \bar{\rho_0}>0,~~~\frac{\p \rho}{\p p}> 0, \q \text{in}\,\, \bar{\Om}.
			\label{EoS}
		\end{align} for some constant $\bar{\rho_0}>0$. For instance, we have ideal fluids $\rho(p,S)=p^{1/\gamma}e^{-S/\gamma}$ with $\gamma>1$ for a polytropic gas. These two conditions also guarantee the hyperbolicity of system \eqref{CElasto}.
		
		The initial and boundary conditions of system \eqref{CElasto} are
		\begin{align}
			(p,u,\dot{\boldsymbol{F}},S)|_{t=0}=(p_0,u_0, \dot{\boldsymbol{F}}_0,S_0) ~~~& \text{in }[0,T]\times\Om,\\
			u\cdot N=0,\quad \dot{\boldsymbol{F}}^{\mathrm{T}}\cdot N=0 ~~~& \text{on }[0,T]\times\Sigma,
		\end{align} 
		where the boundary condition for $u$ is the slip boundary condition. The condition for $\dot{\boldsymbol{F}}$ is not an imposed boundary condition. Instead, this condition is also a constraint for initial data that propagates within the lifespan of the solution and we refer to \cite{Trakhinin2016elastic} for the proof. The degeneracy is related to the formation of vortex sheets in elastodynamics and we refer to \cite[Remark 2.1]{CHWWY2} or \cite[Section 186]{VSphysical} for further interpretations.
		
		Define $\dot{F}_j=(\dot{F}_{1j},\dot{F}_{2j},\dot{F}_{3j})^{\mathrm{T}}$ to be the $j$-th column of $\dot{\boldsymbol{F}}$. Then we have
		\begin{equation}
			\begin{cases}
				D_t\rho+\rho(\nabla\cdot u)=0~~~&\text{in}~[0,T]\times\Om, \\
				\rho D_t u +\eps^{-2} \nabla p= \rho\sum\limits_{j=1}^{3}\dot{F}_j\cdot\nabla \dot{F}_j~~~& \text{in}~[0,T]\times\Om, \\
				D_t \dot{F}_j=\dot{F}_j\cdot\nabla u~~~&\text{in}~[0,T]\times\Om, \\
				\nabla\cdot (\rho \dot{F}_j)=0~~~&\text{in}~[0,T]\times\Om,\\
				D_t S=0~~~&\text{in}~[0,T]\times\Om.
			\end{cases}\label{CElasto1}
		\end{equation}
		The initial and boundary conditions of system \eqref{CElasto1} are
		\begin{align}
			\label{data}	    (p,u,\dot{F}_j,S)|_{t=0}=(p_0,u_0, \dot{F}_{j,0},S_0) ~~~& \text{in }[0,T]\times\Om,\\
			\label{bdry cond}			u_3=0,\quad \dot{F}_{3j}=0 ~~~& \text{on }[0,T]\times\Sigma.
		\end{align}

		To make the initial-boundary value problem \eqref{CElasto1}-\eqref{bdry cond} solvable, we need to require the initial data to satisfy the compatibility conditions up to certain order. For $m\in\N$, we define the $m$-th order compatibility conditions to be
		\begin{equation}\label{comp cond}
			\p_t^j u_3|_{t=0}=0\text{ on }\Sigma,~~0\leq j\leq m.
		\end{equation}

		Let $a:=\frac{1}{\rho}\frac{\p\rho}{\p p}$. Since $\frac{\p \rho}{\p p}>0$ implies $a(p,S)>0$, using $D_t S=0$, the first equation of \eqref{CElasto1} is equivalent to 
		\begin{equation} \label{continuity eq f}
			a D_t p  +\nab\cdot u=0. 
		\end{equation}
		Thus the compressible elastodynamic system is now reformulated as follows
		\begin{equation}\label{CElasto2}
			\begin{cases}
				a D_t p  +\nab\cdot u=0~~~&\text{in}~[0,T]\times\Om, \\
				\rho D_t u +\eps^{-2}\nabla p=\rho\sum\limits_{j=1}^{3}\dot{F}_j\cdot\nabla \dot{F}_j~~~& \text{in}~[0,T]\times\Om, \\
				D_t \dot{F}_j=\dot{F}_j\cdot\nabla u~~~&\text{in}~[0,T]\times\Om, \\
				\nabla\cdot (\rho \dot{F}_j)=0~~~&\text{in}~[0,T]\times\Om,\\
				D_t S=0~~~&\text{in}~[0,T]\times\Om,\\
				a=a(p,S)>0,\quad\rho=\rho(p, S)>0~~~&\text{in}~[0,T]\times\bar{\Om}.\\
				u_3=\dot{F}_{3j}=0 ~~~& \text{on}~[0,T]\times\Sigma,\\
				(p,u,\dot{F}_j,S)|_{t=0}=(p_0,u_0, \dot{F}_{j,0}, S_0)~~~& \text{on}~\{t=0\}\times\Om.
			\end{cases}
		\end{equation}

		When considering the incompressible limit, that is, when $\eps>0$ tends to 0, it is more convenient to symmetrize the compressible elastodynamic system by using the transformation 
		\[
		p=1+\eps q,\quad \dot{F}_j=F_j + \BF_j,\quad \BF_j=(\BF_{1j},\BF_{2j},0)^{\mathrm{T}},
		\]
		where $\BF_{ij}$ are constants, and $F_j$ are functions which represent the perturbation around the constant states. This step is necessary when $\Om$ is unbounded because we want each variable to belong to $L^2(\Om)$. We then derive the following dimensionless non-isentropic inviscid elastodynamic system.
		\begin{equation}\label{CElasto3}
			\begin{cases}
				a D_t q  + \eps^{-1} \nab\cdot u=0~~~&\text{in}~[0,T]\times\Om, \\
				\rho D_t u +\eps^{-1}\nabla q=\rho\sum\limits_{j=1}^{3}(F_j+\BF_j)\cdot\nabla F_j~~~& \text{in}~[0,T]\times\Om, \\
				D_t F_j=(F_j+\BF_j)\cdot\nabla u~~~&\text{in}~[0,T]\times\Om, \\
				\nabla\cdot (\rho (F_j+\BF_j))=0~~~&\text{in}~[0,T]\times\Om,\\
				D_t S=0~~~&\text{in}~[0,T]\times\Om,\\
				a=a(\eps q,S)>0,\quad \rho=\rho(\eps q,S)>0~~~&\text{in}~[0,T]\times\bar{\Om}.\\
				u_3=F_{3j}=0 ~~~& \text{on}~[0,T]\times\Sigma,\\
				(q,u,F_j,S)|_{t=0}=(q_0,u_0, F_{j,0}, S_0)~~~& \text{on}~\{t=0\}\times\Om.\\
			\end{cases}
		\end{equation}

		\subsection{An overview of previous results}\label{sect history}

		The incompressible limit of compressible inviscid fluids is considered to be a type of singular limit for hyperbolic system: the pressure for compressible fluids is a variable of hyperbolic system whereas the pressure for incompressible fluids is a Lagrangian multiplier and the equation of state is no longer valid. Early works about compressible Euler equations can be dated back to Klainerman-Majda \cite{Klainerman1981limit, Klainerman1982limit} when the domain is the whole space $\R^d$ or the periodic domain $\T^d$,  Schochet \cite{Schochet1986limit} when the domain is bounded, and Isozaki \cite{Isozaki1987limit} when considering an exterior domain. See also Sideris \cite{Sideris1991} and Secchi \cite{Secchi2002} for the long-time incompressible limit of Euler equations in $\R^3$ and $\R\times\R_-$ respectively. We also refer to Luo \cite{Luo2018CWW} and Luo and the second author \cite{LuoZhang2022CWWST} for the incompressible limit of free-surface Euler equations with or without surface tension. The abovementioned papers consider the case of ``well-prepared" initial data, that is, $(\nab\cdot u_0,\nab q_0)=O(\eps^k)$ for $k\geq 1$, which means the compressible initial data is exactly a slight perturbation of the given incompressible initial data. In this case, the first-order time derivatives of each variable is bounded uniformly in $\eps$, and so uniform estimates immediately lead to the strong convergence to the incompressible system. 
		
		However, for general initial data (also called ``ill-prepared" initial data), that is, $(\nab\cdot u_0,\nab q_0)=O(1)$, the compressible initial data are no longer a small perturbation of incompressible initial data but also contains a highly oscillatory part. In this case, the first-order time derivatives of velocity and pressure are of $O(\eps^{-1})$ size and one has to filter the highly oscillatory part (actually acoustic waves) when proving the strong convergence. We refer to \cite{Ukai1986limit, Isozaki1987limit, Asano1987limit, Schochet1994limit, Iguchi1997limit, Secchi2000, Cheng2012} for the incompressible limit of isentropic Euler equations with general data and \cite{Metivier2001limit, Alazard2005limit} for the incompressible limit of non-isentropic Euler equations with general initial data in $\R^d$ or the half-space or the exterior of a bounded domain.
		
		For the incompressible limit of inviscid elastodynamic system, when the domain is the whole space or the periodic domain, we refer to \cite{Schochet1985elasticity, Sideris2005, JuWangXu2022, Wang2022elasto, JuWang2023}. In particular, when the domain has a boundary, the existing literature only considers the isentropic case when the deformation tensor satisfies the degenerate constraint $\dot{\boldsymbol{F}}^\mathrm{T}\cdot N=0$ on $\p\Om$, and we refer to Liu-Xu \cite{LiuXu2021} for the case of well-prepared initial data and Ju, the first author and Xu \cite{JuWangXu2022} for the case of general initial data. Recently, the second author \cite{Zhang2021elasto} proved local well-posedness and incompressible limit of the free-boundary problem to the isentropic compressible elastodynamic equations. 
		
		However, the study of incompressible limit for non-isentropic fluid is more subtle. In fact, when the data is well-prepared, the frameworks for the isentropic case are still valid up to some technical modifications. For example, the authors \cite{WangZhang2025JDE} proved the incompressible limit of non-isentropic MHD with well-prepared data by combining the framework of \cite{LuoZhang2022CWWST}  and some observations for MHD equations, in which the weights of Mach number should be carefully chosen according to the number of tangential derivatives such that the energy estimates are uniform in Mach number. Unfortunately, when the initial data is ill-prepared, the simultaneous appearance of compressibility, entropy and the coupled quantities (such as the elasticity, the magnetic fields for MHD equations, etc) causes several essential difficulties that do not appear in Euler equations with general initial data or the coupled system (such as elastodynamics, MHD, etc) with well-prepared initial data. In particular, there is a simultaneous loss of weights of Mach number and derivatives in the vorticity analysis for non-isentropic elastodynamics (also for MHD) with general initial data. On the other hand, the incompressible limit for viscoelastic systems has been extensively studied by many researchers; see, for example, \cite{CuiHu2022ARMA,CuiHu2023Vanishing,HuOuWangYang2023AdvNA,Lei2006global,LeiAndZhou2005global,GuillopeSalloum2010DCDS,FangAndZi2014incompressible,RenAndOu2015strong,ZhangZhong2025JLMS}. 
		
		The aim of this paper is to establish the incompressible limit problem of non-isentropic elastodynamics inside a solid wall with general initial data. There are mainly two important observations which will be discussed in Section \ref{sect overview intro}. It should also be noted that the smallness of Mach number $\eps$ is required to close the uniform estimates in many previous works about the incompressible limits in a domain with boundary, such as \cite{Ukai1986limit, Asano1987limit, Iguchi1997limit, Secchi2000, Schochet1986limit, Metivier2001limit, Alazard2005limit, LiuXu2021, JuWangXu2022}, while \textit{our method (when closing the energy estimates) no longer requires the smallness of Mach number}.
		
		This is a significant improvement because our uniform-in-$\eps$ energy estimates can be closed even when the Mach number $\eps$ is not assumed to be sufficiently small, which in turn implies the local well-posedness of the system using the same energy functional. In contrast, the energy estimates in the aforementioned previous works critically depend on the smallness of $\eps$. Specifically, they typically lead to an inequality of the form 
			\[E(t) \le  P(E(0)) + \boxed{\eps P(E(t))}+ \int_0^t P(E(\tau))\dtau.\]
		When $\eps$ is not small, the term $\eps P(E(t))$ on the right-hand side cannot be absorbed by the $E(t)$ on the left-hand side, thus preventing us closing the energy estimates via Gr\"onwall's inequality. Our approach finally gives the following energy inequality
		\[
			E(t) \le  P(E(0)) + \boxed{\delta E(t)} + P(E(t))\int_0^t P(E(\tau))\dtau,
		\] where $\delta>0$ (\underline{independent of $\eps$}) can be chosen suitably small such that $\delta E(t)$ can be absorbed by the left side. In other words, we successfully avoid the problematic term $\eps P(E(t))$ by exploiting the special structure of the wave equation for the pressure and employing elliptic estimates, thereby enabling a uniform analysis that is valid regardless of the size of $\eps$. See section \ref{sect. 2-2} for details.

		\subsection{The main theorems}		
		We denote the interior Sobolev norm to be 
		$$\|f\|_{s}:= \|f(t,\cdot)\|_{H^s(\Omega)},\quad \|f\|^2_{s,\eps}:=\sum_{k=0}^{s}\|(\eps\p_t)^k f\|_{s-k}^2$$
		for any function $f(t,x)\text{ on }[0,T]\times\Omega$ and denote the boundary Sobolev norm to be $|f|_{s}:= |f(t,\cdot)|_{H^s(\Sigma)}$ for any function $f(t,x)\text{ on }[0,T]\times\Sigma$.

		The local well-posedness of \eqref{CElasto3} in $H^3(\Om)$ for each fixed $\eps>0$ can be proved by using the classical theory for symmetric hyperbolic systems with characteristic boundary, such as \cite{Schochet1986limit} and \cite[Appendix A]{WangZhang2025JDE}. First, we establish a local-in-time estimate, uniform in Mach number $\eps$, without assuming $\eps$ to be small.
		\begin{thm}[\textbf{Uniform-in-$\eps$ estimate}]\label{main thm, ill data}
			Let $\eps>0$ be given. Let $(q_0, u_0,F_{j,0},S_0)\in H^3(\Om)\times H^3(\Om)\times H^3(\Om)\times H^4(\Om)$ be the initial data of \eqref{CElasto3} satisfying the compatibility conditions \eqref{comp cond} up to 2nd order and
			\begin{equation}
				E(0)\le M
			\end{equation}
			for some $M>0$ independent of $\eps$. Then there exist $T>0$ and $\eps_0\in (0,1)$ depending only on $M$, such that for all $\eps\in(0,\eps_0)$, \eqref{CElasto3} admits a unique solution $(q(t),u(t),F_j(t),S(t))$ that satisfies the energy estimate
			\begin{equation}\label{energy estimate}
				\sup_{t\in[0,T]}E(t) \le P(E(0)),
			\end{equation}
			where $P(\cdots)$ is a generic polynomial in its arguments, and the energy $E(t)$ is defined to be
			\begin{equation}
				\begin{aligned}\label{energy intro}
					E(t)=\left\|(q,u,S)\right\|_{3,\eps}^2+\sum_{j=1}^3\left\|\left(F_j,(F_j+\BF_j)\cdot\nabla S\right)\right\|_{3,\eps}^2.
				\end{aligned}
			\end{equation}
		\end{thm}
		\begin{rmk}[Enhanced ``directional" regularity of the entropy]
			The assumption $S_0\in H^4(\Om)$ is imposed in order that $(F_{j,0}+\BF_{j,0})\cdot\nab S_0$ belongs to $H^3(\Om)$. In this paper, we only need such enhanced regularity of $S$ in the direction of $F_j+\BF_j~(j=1,2,3)$ instead of the full $H^4$ regularity. One can prove that the solution also satisfies $\fjp S\in H^3(\Om)$ as long as the initial data satisfies, and we refer to Corollary \ref{lem_EnergyEsti_S} for the proof.
		\end{rmk}

		The next main theorem concerns the incompressible limit. We consider the incompressible inhomogeneous elastodynamic equations satisfied by $(u^0,F_j^0,\pi,S^0)$:
		\begin{equation} \label{IElasto}
			\begin{cases}
				\p_t\varrho +u^0\cdot\nab \varrho=0,\quad \varrho=\rho(0,S^0)&~~~ \text{in}~[0,T]\times \Omega,\\
				\varrho(\p_t u^0 + u^0\cdot\nab u^0) + \nab \pi =\varrho \sum\limits_{j=1}^{3}(F_j^0+\BF_j)\cdot\nabla F_j^0&~~~ \text{in}~[0,T]\times \Omega,\\
				\p_t F_j^0+ u^0\cdot\nab F_j^0=(F_j^0+\BF_j)\cdot\nab u^0 &~~~ \text{in}~[0,T]\times \Omega,\\
				\nab\cdot u^0=\nab\cdot (\varrho( F_j^0+\BF_j ))=0&~~~ \text{in}~[0,T]\times \Omega,\\
				u_3^0=F_{3j}^0=0&~~~\text{on}~[0,T]\times\Sigma.
			\end{cases}
		\end{equation}

		\begin{thm}[\textbf{Incompressible limit}] \label{main thm, limit}
			Under the hypothesis of Theorem \ref{main thm, ill data}, we assume that $(u_0,F_{j,0},S_0) \to (u_0^0,F_{j,0}^0,S_0^0)$ in $H^3(\Omega)$ as $\eps\to 0$ with $\nab\cdot ( \rho(0,S_0^0) F_{j,0}^0 )=0$ in $\Om$ and $u_{03}^0=F_{3j,0}^0=0$ on $\Sigma$, and that there exist positive constants $N_0$ and $\sigma$ such that $S_0$ satisfies
			\begin{equation}\label{EntropyDecay}
				|S_0(x)| \le N_0 |x|^{-1-\sigma},\quad |\nabla S_0(x)|\le N_0|x|^{-2-\sigma}.
			\end{equation}
			Then it holds that 
			\begin{align*}
				(q,u,F_j,S) \to (q^0,u^0,F_j^0,S^0) &\text{ weakly-* in } L^{\infty}([0,T];H^3(\Omega)),\\
				(F_j,S)\to (F^0_j,S^0) &\text{ strongly in } C([0,T];H^{3-\delta}_{\mathrm{loc}}(\Omega)),\\
				(q,u)\to (0,u^0) &\text{ strongly in } L^2([0,T];H^{3-\delta}_{\mathrm{loc}}(\Omega)),
			\end{align*}
			for $\delta>0$. $(u^0,F_j^0,S^0)\in C([0,T];H^3(\Om))$ solves \eqref{IElasto} with initial data $(u^0,F_j^0,S^0)|_{t=0}=(w_0,F_{j,0}^0,S_0^0)$, that is, the incompressible inhomogeneous elastodynamic equations, where $w_0\in H^3(\Omega)$ is determined by
			\begin{equation}
				w_{03}|_{\Sigma}=0,\quad \nabla\cdot w_0=0,\quad \nabla\times (\rho(0,S^0_0)w_0 )=\nabla\times(\rho(0,S_0^0)u_0^0).
			\end{equation}
			The function $\pi$ satisfying $\nab\pi\in C([0,T];H^2(\Om))$ represents the fluid pressure for incompressible elastodynamic system \eqref{IElasto}.
		\end{thm}
		
		\begin{rmk}[Unboundedness of the domain]
			In Theorem \ref{main thm, ill data}, the uniform-in-$\eps$ estimate can be established regardless of the boundedness of $\Om$. The unboundedness of $\Om$ is required in the proof of strong convergence. In fact, the strong convergence in time can be obtained by proving local energy decay due to the (global) dispersion property for the wave equation of the pressure as in \cite{Metivier2001limit, Alazard2005limit}, in which the unboundedness of $\Om$ and the entropy decay condition \eqref{EntropyDecay} are both needed.
		\end{rmk}		
		
		\begin{rmk}[The case of domains with curved boundaries]
			We choose $\Om=\R^2\times\R_-$ for technical simplicity as its boundary is flat, but our conclusion still holds for some unbounded domain with a curved boundary, for example, the case that $\Omega$ is the exterior of a bounded domain with an $H^{3.5}$ boundary $\Sigma$, as shown in Alazard \cite{Alazard2005limit}. We note that such regularity of $\Sigma$ is required to ensure the div-curl inequality in Lemma \ref{lem divcurl}, according to \cite[Theorem 1.1(2)]{CSdivcurl}. In such case, $\Omega$ has a finite covering such that
			$$
			\Omega \subset \Omega_0 \bigcup\left(\bigcup_{i=1}^m \Omega_i\right), \quad \Omega_0 \Subset \Omega, \quad \Omega_i \cap \Sigma \neq \emptyset,
			$$
			and $\Omega_i \cap \Sigma$ is the graph of a smooth function $z=\varphi_i\left(x_1, x_2\right)$. We use the local coordinates in each $\Omega_i$, $i=1,2, \cdots, m$ :
			$$
			\begin{aligned}
				\Phi_i:(-1,1)^2 \times(-1,0) & \rightarrow \Omega_i \cap \Omega \\
				(y, z)^{\mathrm{T}} & \rightarrow \Phi_i(y, z)=\left(y, \varphi_i(y)+z\right)^{\mathrm{T}} .
			\end{aligned}
			$$
			
			We denote by $N$ the unit outward normal to the boundary. In each $\Omega_i$, we can extend it to $\Omega_i$ by setting
			$$
			N(x):=N\left(\Phi_i(y, z)\right)=\left(1+\left|\cnab \varphi_i(y)\right|^2\right)^{-1 / 2}\left(-\partial_1 \varphi_i(y), -\partial_2 \varphi_i(y),1\right)^{\mathrm{T}},\q\q \cnab:=(\p_1,\p_2)^{\mathrm{T}} .
			$$
			In such case, the basic states $\bar{F_j}$ can be suitably chosen smooth functions, not necessarily constants. For example, we define the matrix function $\bar{F}(x)=\left(\bar{F}_1, \bar{F}_2, \bar{F}_3\right)(x)$ as
			$$
			\bar{F}(x)=\left\{\begin{array}{lll}
				I_3, & x \in \Omega_0, \\
				\bar{F}\left(\Phi_i(y, z)\right)=\left(\begin{array}{ccc}
					1 & 0 & 0 \\
					0 & 1 & 0 \\
					\partial_1 \varphi_i(y) & \partial_2 \varphi_i(y) & \phi(z)
				\end{array}\right), & x \in \Omega_i \cap \Omega,
			\end{array} \quad \phi(z)=\frac{-z}{1-z} .\right.
			$$
			
			When $\Omega$ is the half space $\mathbb{R}_{-}^3$, one of the choices for a non-constant $\bar{F}$ would be
			$$ 
			\bar{F}(x)=\left(\begin{array}{ccc}
				1 & 0 & 0 \\
				0 & 1 & 0 \\
				0 & 0 & \phi\left(x_3\right)
			\end{array}\right), \quad \phi\left(x_3\right)=\frac{-x_3}{1-x_3} .
			$$
			
			It is easy to verify that
			$$
			\begin{array}{cc}
				\operatorname{det} \bar{F}\neq 0 & \text { in }[0, T] \times \Omega, \\
				\bar{F}_j \cdot N=0 & \text { on }[0, T] \times \Sigma .
			\end{array}
			$$
		\end{rmk}		
		
		\subsection{Organization of the paper}\label{sect org} 
		This paper is organized as follows. In Section \ref{sect overview intro}, we discuss the major difficulties in this problem. Then Section \ref{sect uniform} is devoted to the proof of uniform estimates in Mach number. The strong convergence for the incompressible limit problem is proved in Section \ref{sect limit}.  In Appendix \ref{sect lemma}, we record several lemmas that are repeatedly used throughout this manuscript.

		\subsection*{List of Notations}
		\begin{itemize}
			\item ($L^\infty$-norm) $\|\cdot\|_{\infty}:= \|\cdot\|_{L^\infty(\Omega)}$, {$|\cdot|_{\infty}:=\|\cdot\|_{L^\infty(\Sigma)}$}. 
			\item (Interior Sobolev norm) $\|\cdot\|_{s}$:  We denote $\|f\|_{s}:= \|f(t,\cdot)\|_{H^s(\Omega)}$ and $\|f\|_{s,\eps}^2=\sum_{k=0}^{s}\|(\eps\p_t)^k f\|_{s-k}^2$ for any function $f(t,y)\text{ on }[0,T]\times\Omega$.
			\item  (Boundary Sobolev norm) $|\cdot|_{s}$:  We denote $|f|_{s}:= |f(t,\cdot)|_{H^s(\Sigma)}$ for any function $f(t,y)\text{ on }[0,T]\times\Sigma$.
			\item (Polynomials) $P(\cdots)$ denotes a generic polynomial in its arguments.
			\item (Commutators) $[T,f]g=T(fg)-f(Tg)$, $[f,T]g=-[T,f]g$ where $T$ is a differential operator and $f,g$ are functions.
			\item (Leray projection operator) Consider the orthogonal decomposition $L^2(\Omega)=H_{\sigma} \oplus G_{\sigma}$ with $H_\sigma=\{u\in L^2(\Omega): \int_{\Om} u\cdot\nabla\phi,\ \forall \phi\in H^1(\Omega)\}$ and $G_{\sigma}=\{\nabla\psi: \psi\in H^1(\Omega)\}$. Let $\mathcal{P}$ be the projection onto $H_{\sigma}$ and $\mathcal{Q}= I_3 - \mathcal{P}$.
		\end{itemize}

		\section{Difficulties and strategies}\label{sect overview intro}
		
		System \eqref{CElasto} is symmetric hyperbolic with characteristic boundary conditions \cite{Rauch1985} and so there is a potential to have loss of normal derivatives, which is expected to be compensated by using div-curl analysis. Indeed, major difficulties in this problem exactly appear in the proof of div-curl estimates. Once we establish the control for the divergence and the vorticity, the uniform estimates can be closed by controlling the full-time derivatives which is parallel to the $L^2$ estimate. After that, the strong convergence to the incompressible  system can be proved via a slight variant of the argument in \cite{Alazard2005limit}.
		
		Below, we briefly discuss our observations that are used to overcome the main difficulties that do not appear in the study of Euler equations or isentropic fluids. 
		
		\subsection{Observation 1: Enhanced ``directional" regularity of the entropy}
		
		Since we are considering the non-isentropic case with general initial data, inspired by \cite{Metivier2001limit}, we shall rewrite the momentum equation as
		\begin{equation}
			D_t(\rho_0 u) -\rho_0 \sum_{j=1}^{3}(F_j+\BF_j)\cdot\nabla F_j=-\eps^{-1}\nabla q  + \frac{\rho-\rho_0}{\eps \rho}\nabla q
		\end{equation}
		with $\rho_0=\rho(0,S)$ to analyze the vorticity. Then the evolution equation of the vorticity becomes
		\begin{equation}
			D_t(\nabla\times(\rho_0 u))- \sum_{j=1}^{3}\nabla\times( \rho_0 (F_j+\BF_j)\cdot\nabla F_j  )- \nabla\left( \frac{\rho-\rho_0}{\eps \rho} \right)\times\nabla q=\text{ controllable terms}.
		\end{equation} The reason for replacing $\rho$ by $\rho_0$ is that $D_t \rho_0=0$ allows us to ``hide" this coefficient into $D_t$, otherwise, when taking derivatives $\p^\alpha$ in the momentum equation, there must be terms like $\p^\alpha \rho D_t u$ appearing without any $\eps$-weight. When $\p^\alpha$ falls on $S$ in $\rho=\rho(\eps q,S)$, $\p^\alpha\rho$ must generate an $O(1)$-size term in front of $D_t u$ and thus there exhibits a loss of weights of Mach number due to the ill-preparedness of initial data $(\p_t u = O(1/\eps))$.
		
		However, for elastodynamics, the simultaneous appearance of the deformation tensor, the non-constant entropy and compressibility leads to an extra loss of derivative in vorticity analysis. For example, in the $H^2$-control of $\nabla\times(\rho_0 u)$, we must encounter the following underlined terms
		\begin{align}
			&\ddt \int_{\Om}|\p^2\nabla\times(\rho_0 u)|^2 + \sum_{j=1}^{3}|\p^2\nabla\times(\rho_0 F_j)|^2 \notag\\
			=&-\sum_{j=1}^{3}\int_{\Om} (\underline{  \p^2\nabla( (F_j+\BF_j)\cdot \nabla\rho_0 } )\times F_j  )\cdot \p^2\nabla\times (\rho_0 u) \dx\notag\\
			&-\sum_{j=1}^{3}\int_{\Om} (  \underline{\p^2\nabla( (F_j+\BF_j)\cdot \nabla\rho_0  })\times u  )\cdot \p^2\nabla\times (\rho_0 F_j) \dx + \text{controllable terms},
		\end{align}
		where we find that there are 4 derivatives falling on $\rho_0$ (equivalently, on $S$) and thus the vorticity estimates cannot be closed in the setting of third-order Sobolev spaces. To overcome such difficulty, we just need a rather simple observation: the material derivative commutes with the directional derivatives $\fjp$ for $j=1,2,3$, which is even easier to observe if one uses Lagrangian coordinates (to study the free-boundary problems) as in the second author's work \cite{Zhang2021elasto} because $D_t$ becomes $\p_t$ and each $\fjp$ becomes time-independent in Lagrangian coordinates. Thus, $(F_j+\BF_j)\cdot \nabla S$ also has $H^3(\Om)$ regularity as long as its initial data belong to $H^3(\Om)$. That is exactly why we impose such enhanced regularity for the initial data $S_0$. The details are referred to Section \ref{sect entropy} and \ref{sect curl}. To our knowledge, such observation does not appear in previous works, but it can really help us prove the uniform estimates in Mach number for the vorticity part.  
		
		\subsection{Observation 2: Structure of the wave equation of $q$}\label{sect. 2-2}
		
		As stated at the end of Section \ref{sect history}, the smallness of Mach number $\eps$ is required to prove the uniform estimates in previous works about incompressible limit of non-isentropic inviscid fluids in a domain with boundary, because the divergence part, namely $\nab\cdot u \approx -\eps D_t q$, essentially contributes to $\eps P(E(t))$. For elastodynamics, we also have $\nab\cdot F_j\approx -\eps \fjp q$ which still contributes to $\eps P(E(t))$. Then the smallness of $\eps$ can be used to absorb such $\eps P(E(t))$ terms in the Gr\"onwall-type inequality $E(t)\lesssim  P(E(0))+\eps P(E(t))  + \int_0^t P(E(\tau))\dtau$. 
		
		In this paper, we aim to drop the smallness assumption of $\eps$, that is, our energy estimates are uniform in $\eps$ even if the Mach number $\eps$ is not small. This means that the low Mach number limit automatically holds if we use the same energy to prove the local existence, for example, by standard Picard iteration in \cite[Appendix A]{WangZhang2025JDE}. Therefore, we must seek for a different way to control the divergence part.
		
		Note that the pressure $q$ satisfies a wave equation
		\begin{align}
			\eps^2 a D_t^2 q - \nab \cdot (\rho^{-1}\nab q) - \sum_{j=1}^3 \eps^2 a (\fjp)^2 q = \mathcal{G}^\eps
		\end{align}with Neumann boundary condition $\p_3 q = 0$ on $\Sigma$ and $\|\eps^{-1} \mathcal{G}^\eps\|_2$ controlled by $P(E(t))$ uniformly in $\eps$. It is easy to observe that standard wave-type estimate already gives us the uniform $L^2(\Om)$ control of $\eps D_t q$, $\eps\fjp q$ and $\nab q$. To control high-order Sobolev norms, we must  reduce the normal derivatives to tangential derivatives as taking normal derivatives does not preserve the Neumann boundary condition. This can be done by combining the div-curl inequality (Lemma \ref{lem divcurl}) and the concrete form of the wave equation. For example, we have
		\[
		\|\nab q\|_2^2\lesssim \|\nab q\|_0^2+\|\lap q\|_1^2+\underbrace{\|\nab\times\nab q\|_1^2+|\p_3 q|_{1.5}^2}_{=0},
		\]and then $\lap q$ can be converted to $\eps^2 D_t^2 q$ and $\eps^2(\fjp)^2 q$, which are tangential derivatives, plus lower-order terms that are easy to control. We can do similar things for $\eps D_t q$ and $\eps \fjp q$
		\begin{align*}
			\|\eps \nab D_tq\|_2^2\lesssim&~ \|\eps\nab D_t q\|_0^2+\|\eps \lap D_t q\|_0^2+\underbrace{\|\eps \nab\times\nab D_t q\|_0^2+|\eps D_t\p_3 q|_{0.5}^2}_{=0}+\underbrace{|\eps [D_t,\p_3]q|_{\frac12}^2}_{\text{lower order}},\\
			\|\eps \nab \fjp q\|_2^2\lesssim&~ \|\eps\nab \fjp q\|_0^2+\|\eps \lap \fjp q\|_0^2\\
			&+\underbrace{\|\eps \nab\times\nab \fjp q\|_0^2+|\eps \fjp\p_3 q|_{0.5}^2}_{=0}+\underbrace{|\eps [\fjp,\p_3]q|_{\frac12}^2}_{\text{lower order}}.
		\end{align*}Therefore, the control of $\eps\lap D_t q, \eps\lap \fjp q$ and $\lap q$ can be further converted to that of 
		\begin{equation}\label{reduce div intro}
			\begin{aligned}
				&\|\eps^3 D_t^3 q\|_0^2,~\|\eps^2 \nab D_t^2 q\|_0^2~~~(l=1,2,3),\\
				&\|\eps^2\nab D_t\flp  q\|_0^2,~\|\eps^3D_t^2\flp q\|_0^2,\\
				&\sum_{j=0}^3\|\eps^3 D_t (\fjp)^2 q\|_0^2,~\sum_{j=0}^3\|\eps^3  (\fjp)^2\flp q\|_0^2~~~(l=1,2,3).
			\end{aligned}
		\end{equation} The above quantities can all be controlled via the wave equation of $q$ differentiated by {\it tangential derivatives} $D_t^k(\fjp)^{2-k} q$ for $k=0,1,2$ with certain $\eps$ weights thanks to the degeneracy of $F_{3j}$ on $\Sigma$. Those lower order terms generated in the reduction process can all be controlled by repeatedly using multiplicative Sobolev inequality and Young's inequality. We refer to Section \ref{sect div} for the detailed analysis.

		\section{Uniform energy estimates}\label{sect uniform}
		
		Using Lemma \ref{lem divcurl}, we have for $k=0,1,2$ that 
		\begin{align}\label{E_reduce}
			\|(\eps\p_t)^k(u, F_j)\|_{3-k}^2\lesssim &~\|(\eps\p_t)^k(u, F_j)\|_0^2 + \|(\eps\p_t)^k(\nab\cdot u, \nab\times F_j)\|_{2-k}^2 \notag\\
			&+  \|(\eps\p_t)^k(\nab\cdot u, \nab\times F_j)\|_{2-k}^2 + |(\eps\p_t)^k(u_3,F_{3j})|_{\frac52 - k}^2.
		\end{align}Since $u_3=F_{3j}=0$ on $\Sigma$ eliminates the boundary terms, the above inequality motivates us to define $E_1(t)$ and $E_2(t)$ as follows:
		\begin{align}
			\label{E1_energy}	E_1(t)=&\sum_{k=0}^{3}\|(\eps\p_t)^k(q,u,F_j)\|^2_0  + \|S\|_{3,\eps}^2 + \|(F_j+\BF_j)\cdot\nabla S\|_{3,\eps}^2+\|\nabla\times (\rho_0 u)\|_{2,\eps}^2 + \|\nabla\times (\rho_0 F_j)\|_{2,\eps}^2,\\
			\label{E2_energy}	E_2(t)=&~\|\nabla q\|_{2,\eps}^2 + \|\nabla\cdot u\|_{2,\eps}^2 + \|\nabla\cdot F_j\|_{2,\eps}^2 +\|\nabla\times u\|_{2,\eps}^2 + \|\nabla\times F_j\|_{2,\eps}^2,
		\end{align}
		with $\rho_0=\rho(0,S)$. It should be noted that we impose the curl of $\rho_0 u$ and $\rho_0 F$ instead of that of $u$ and $F$ in $E_1(t)$. This substitution is necessary to overcome some technical difficulties in vorticity analysis for the non-isentropic problems with general initial data. 
		
		In order to show the uniform-in-$\eps$ estimates \eqref{energy estimate} holds, it suffices to find norms $E_1(t)$ and $E_2(t)$ satisfying
		\begin{equation}\label{equiv norm}
			E(t) \le C (E_1(t) + E_2(t)),\q\q\text{ for some }C>0\text{ independent of }\eps
		\end{equation}
		and the following  uniform-in-$\eps$ control
		\begin{align}
			&\ddt E_1(t) \le P(E(t)),\label{E1_estimate}\\
			&E_2(t)\le P(E_1(t)) + \delta E(t) + P(E(0)) + P(E(t))\int_0^t P(E(\tau))\dtau,\label{E2_estimate}
		\end{align}for any constant $\delta\in(0,1)$, which then leads to our desired estimates by using Gr\"onwall-type inequality.
		
		Obviously, \eqref{equiv norm} holds true for the above norms thanks to \eqref{E_reduce}. Hence, the rest part of this section is devoted to deriving estimates \eqref{E1_estimate} and \eqref{E2_estimate}. 
		
		\subsection{$L^2$ estimate}\label{sect L2}
		First, it is easy to prove the $L^2$ energy estimate for the dimensionless elastodynamic system \eqref{CElasto3}.
		\begin{prop}\label{L2 energy}
			Define the $L^2$ energy of system \eqref{CElasto3} by
			\begin{align}
				E_0(t):=\frac12\io\rho|u|^2+\sum_{j=1}^3\rho|F_j|^2 +\rho S^2 + a q^2\dx.
			\end{align}Then it satisfies
			\begin{align}\label{L2 estimate}
				\frac{\mathrm{d} E_0(t)}{\dt}\leq P(E(t)).
			\end{align}
		\end{prop}
		\begin{proof}
			Using the continuity equation, we know for each function $f\in L^2(\Om)$ satisfying $D_t f\in L^2(\Om)$, the Reynolds transport formula holds
			\begin{align}\label{transpt formula}
				\frac12\ddt\io\rho |f|^2\dx = \io \rho (D_t f) f\dx.
			\end{align}Thus, we have $\frac12\ddt\io\rho S^2\dx=\io\rho(D_t S)S\dx=0$ and 
			\begin{align*}
				\ddt\frac12\io \rho|u|^2\dx=\io\rho D_tu \cdot u\dx=\sum_{j=1}^3\io \left(\rho \fjp F_j\right)\cdot u\dx - \frac{1}{\eps}\io u\cdot\nab q\dx.
			\end{align*}Integrating by parts and using $\BF_{3j}=F_{3j}=u_3=0$ on $\Sigma$, we get
			\begin{align*}
				\ddt\frac12\io \rho|u|^2\dx=&-\sum_{j=1}^3\rho F_j\cdot\underbrace{\left(\fjp u\right)}_{=D_t F_j} + \underbrace{\nab\cdot(\rho(F_j+\BF_j))}_{=0}\,F_j\cdot u\dx+\frac{1}{\eps}\io \underbrace{(\nab\cdot u)}_{=-\eps a D_t q} q\dx\\
				=&-\frac12\ddt\io \sum_{j=1}^3\rho |F_j|^2+ a q^2 \dx + \frac12\io (D_t a -(u\cdot\nab)a + \nab\cdot(au)) q^2\dx.
			\end{align*} Since $D_t a (\eps q , S)=\eps D_t q\frac{\p a}{\p p} + D_t S \frac{\p a}{\p S}=\eps D_t q\frac{\p a}{\p p} $ has no loss of $\eps$-weight, we know
			\[
			\frac{\mathrm{d} E_0(t)}{\dt}\lesssim \|q\|_{0}^2\left(\|\eps D_t q\|_{\infty} + \|a\|_{W^{1,\infty}(\Om)}\|u\|_{W^{1,\infty}(\Om)}\right)\leq P(E(t))
			\]as desired.
		\end{proof}
		\subsection{Estimates of entropy and its enhanced directional regularity}\label{sect entropy}
		The entropy has enhanced regularity in the direction of $F_j+\BF_j$ for each $j\in\{1,2,3\}$.
		\begin{lem}\label{lem_Trans Equ}
			For system \eqref{CElasto3}, we have $[D_t, \fjp]=0$ for $j=1,2,3$. In particular, this leads to
			\begin{equation}
				D_t( (F_j+\BF_j)\cdot\nabla S)=0.
			\end{equation}
		\end{lem}
		\begin{proof}
			For any function $f$, we compute that,
			\begin{align*}
				[D_t, \fjp]f=&~ \p_t (\fjp f)+u\cdot\nab(\fjp f) - \fjp(\p_t f+u\cdot\nab f)\\
				=&~(\p_t \fj +u\cdot\nab \fj)\cdot\nab f-(\fjp u)\cdot\nab f \\
				=&~(D_t F_j - \fjp u)\cdot \nab f=0.
			\end{align*} 
			In particular, $D_t S=0$ then leads to $D_t(\fjp S)=0$.
		\end{proof}

		Since $D_t S=0$ and $D_t((F_j+\BF_j)\cdot\nabla S)=0$, we can easily prove the estimates for $S$ and $(F_j+\BF_j)\cdot\nabla S$ by directly commuting $D_t$ with $\eps^k\p_t^k\p^{\alpha}$ for $k=1,2,3$ and using \eqref{transpt formula}. The proof does not involve any boundary term because we do not integrate by parts.
		\begin{cor}\label{lem_EnergyEsti_S}
			Under the assumptions of Theorem \ref{main thm, ill data}, we have
			\begin{equation}
				\ddt \left(  \|S\|_{3,\eps}^2 + \|(F_j+\BF_j)\cdot\nabla S\|_{3,\eps}^2  \right) \le P(E(t)).
			\end{equation}
		\end{cor}
		From this corollary, we see why we require $\fjp S$ has the same regularity as $S$. In fact, this corollary plays a significant role in vorticity analysis.
		
		\subsection{Vorticity analysis}\label{sect curl}
		In this section, we aim to establish the control of  $\nabla\times(\rho_0 u)$ and $\nabla\times(\rho_0 F_j)$ and also their time derivatives.
		\begin{lem}\label{lem_EnergyEsti_curl}
			Under the assumptions of Theorem \ref{main thm, ill data}, we have
			\begin{equation}
				\ddt \left(\|\nabla\times(\rho_0 u)\|_{2,\eps}^2 +\sum_{j=1}^{3}\|\nabla\times(\rho_0 F_j)\|_{2,\eps}^2 \right) \le P(E(t)).
			\end{equation}
		\end{lem}
		\begin{proof}
			The momentum equation $\rho D_t u +\eps^{-1}\nabla q =\rho\sum_{j=1}^{3}(F_j+\BF_j)\cdot\nabla F_j  $ can be rewritten as
			\begin{equation}\label{MomentumEqu2}
				D_t(\rho_0 u) -\rho_0 \sum_{j=1}^{3}(F_j+\BF_j)\cdot\nabla F_j=-\frac{1}{\eps}\nabla q  + \frac{\rho-\rho_0}{\eps \rho}\nabla q.
			\end{equation}
			There exists a smooth function $g$ such that
			\[
			\frac{\rho-\rho_0}{\rho}=\eps g(\eps q,S),\quad \|g\|_{3,\eps} \le P(E(t)).
			\]
			We take $\nab\times$ in \eqref{MomentumEqu2} to get the evolution equation
			\begin{align}\label{curlv eq}
				D_t( \nabla\times(\rho_0 u) ) -\sum_{j=1}^{3}\nabla\times\left( \rho_0 (F_j+\BF_j)\cdot\nabla F_j \right) =[D_t,\nabla\times](\rho_0 u) + \nabla g \times\nabla q,
			\end{align}where we notice that the right side only contains the first-order derivatives and does not lose Mach number weight. Note that the equation of state is smooth, so $\p_q \rho$ and $\p_S \rho$ are bounded.

			Since $\|f\|^2_{2,\eps}=\sum_{k=0}^{2}\|(\eps\p_t)^k f\|_{2-k}^2$, we first prove the case when $k=0$. Indeed, the cases for $k=1,2$ follow in the same manner. In order to control the $H^2$ norms of the curl part, we take $\p^2$ in \eqref{curlv eq} to get 
			\begin{align}\label{curlv H3 eq}
				D_t(\p^2 \nabla\times(\rho_0 u) ) -\sum_{j=1}^{3}\p^2\nabla\times\left( \rho_0 (F_j+\BF_j)\cdot\nabla F_j \right)=   \underbrace{\p^2(\text{RHS of }\eqref{curlv eq})+[D_t,\p^2](\nab\times (\rho_0 u)) }_{\RR_1},
			\end{align}where the order of derivatives on the right side must be $\leq 3$. Now, standard $L^2$-type estimate yields that
			\begin{align}
				&\frac{1}{2}\ddt \int_{\Omega} |\p^2\nab\times (\rho_0 u)|^2 \dx=\int_{\Omega} (\p_t \p^2 \nab\times (\rho_0 u) )\cdot \p^2 \nab\times (\rho_0 u) \dx  \notag\\
				=&\int_{\Om} D_t \p^2 \nab\times (\rho_0 u) \cdot\p^2\nab\times (\rho_0 u) \dx \underbrace{- \int_{\Om}(u\cdot\nab) \p^2\nab\times (\rho_0 u)\cdot \p^2\nab \times (\rho_0 u) \dx }_{:=\mathcal{I}_1},
			\end{align}where $\mathcal{I}_1$ can be directly controlled by integrating by parts and using the symmetry
			\begin{align}
				|\mathcal{I}_1|=\frac12\left|\io(\nab\cdot u)|\p^2\nab \times (\rho_0 u)|^2\dx\right| \leq P(E(t)).
			\end{align}
			
			Then invoking \eqref{curlv H3 eq} gives us the following terms
			\begin{align}
				&\int_{\Om}D_t \p^2 \nab\times (\rho_0 u) \cdot\p^2\nab\times (\rho_0 u) \dx\notag\\
				 =&\int_{\Om}\sum_{j=1}^{3} \p^2 \nabla\times\left( \rho_0 (F_j+\BF_j)\cdot\nabla F_j \right) \cdot \p^2 \nab\times (\rho_0 u) \dx \underbrace{+ \int_{\Om}  \RR_1\cdot\p^2 \nab\times (\rho_0 u)\dx }_{:=\mathcal{I}_2}\notag\\
				=&\int_{\Om}\sum_{j=1}^{3}\left(((F_j+\BF_j)\cdot\nab) \p^2 \nab\times (\rho_0 F_j)\right)\cdot\p^2 \nab\times (\rho_0 u) \dx +\mathcal{I}_2 \notag\\
				&\underbrace{+\int_{\Om}\sum_{j=1}^{3} \left( [\p^2\nabla\times,(F_j+\BF_j)\cdot\nabla](\rho_0 F_j) - [\p^2\nabla\times, F_j]((F_j+\BF_j)\cdot\nabla\rho_0)   \right)\cdot\p^2 \nab\times (\rho_0 u) \dx}_{:=\mathcal{I}_3}  \notag\\
				&\underbrace{-\int_{\Om}\sum_{j=1}^{3} \left( \p^2\nabla( (F_j+\BF_j)\cdot\nabla\rho_0 ) \times F_j   \right)\cdot\p^2 \nab\times (\rho_0 u) \dx}_{:=\mathcal{I}_4}  .
			\end{align}
			Now, we integrate by parts the tangential derivative $\fjp$ to get 
			\begin{align}
				&\int_{\Om}\sum_{j=1}^{3}\left(((F_j+\BF_j)\cdot\nab) \p^2 \nab\times (\rho_0 F_j)\right)\cdot\p^2 \nab\times (\rho_0 u) \dx \notag \\
				=&-\int_{\Om}\sum_{j=1}^{3}(\p^2\nab\times (\rho_0 F_j)) \cdot \p^2\nab\times (((F_j+\BF_j)\cdot\nab) (\rho_0u)) \dx \notag\\
				&\left.\begin{aligned}
					-\int_{\Om}\sum_{j=1}^{3} \nab\cdot \left( F_j+\BF_j \right) (\p^2 \nab\times (\rho_0 F_j)) \cdot (\p^2 \nab\times (\rho_0 u)) \dx\\
					- \int_{\Om}\sum_{j=1}^{3}(\p^2 \nab\times (\rho_0 F_j)) \cdot[(F_j+\BF_j) \cdot\nab, \p^2 \nab\times](\rho_0u) \dx
				\end{aligned}\right\}:=\mathcal{I}_5\notag\\
				=&-\int_{\Om}\sum_{j=1}^{3}(\p^2\nab\times (\rho_0 F_j)) \cdot \p^2\nab\times ( \rho_0 (F_j+\BF_j)\cdot\nab u ) \dx +\mathcal{I}_5\notag\\
				&\underbrace{-\int_{\Om} \sum_{j=1}^{3} (\p^2\nab\times (\rho_0 F_j)) \cdot \left( \p^2\nabla((F_j+\BF_j) \cdot\nabla \rho_0) \times u  + [\p^2\nabla\times, u]((F_j+\BF_j)\cdot\nabla \rho_0)  \right) \dx }_{:=\mathcal{I}_6}
			\end{align}
			Next, we insert $D_t F_j=(F_j+\BF_j)\cdot\nabla u$ to get
			\begin{align}
				&-\int_{\Om}\sum_{j=1}^{3}(\p^2\nab\times (\rho_0 F_j)) \cdot \p^2\nab\times (\rho_0 (F_j+\BF_j)\cdot\nab u) \dx\notag\\
				 =&-\int_{\Om}\sum_{j=1}^{3}\p^2 \nab\times (\rho_0 F_j) \cdot \p^2 \nab\times D_t (\rho_0 F_j) \dx  \notag\\
				=&-\frac{1}{2}\ddt\int_{\Omega} \sum_{j=1}^{3}|\p^2\nab\times (\rho_0 F_j)|^2 \dx \underbrace{- \int_{\Om} \p^2\nab\times (\rho_0 F_j) \cdot ([\p^2\nab\times, D_t](\rho_0 F_j) + (u\cdot\nab)\p^2\nab\times (\rho_0 F_j) ) \dx}_{:=\mathcal{I}_7}.
			\end{align}
			
			Based on the concrete forms of the commutators in Lemma \ref{commutators}, a straightforward product estimate for $\mathcal{I}_j$ $(2\le j\le 7)$ and the estimate for $\mathcal{I}_1$ gives us
			\begin{align}
				\sum_{j=1}^{7}\mathcal{I}_j \le P(E(t)),
			\end{align}
			which gives us the energy estimate
			\begin{align}
				\ddt \left(\|\nab\times (\rho_0 u)\|_2^2 +\sum_{j=1}^{3}\|\nab\times (\rho_0 F_j)\|_2^2\right)\leq P(E(t)).
			\end{align}

			Similarly, we can prove the same conclusion for $\p^{\alpha}(\eps\p_t)^k$ with $k+|\alpha|=2$ by replacing $\p^2$ with $\p^{\alpha}(\eps\p_t)^k$. Indeed, the highest order derivatives in the above commutators do not exceed third order, and there is no loss of Mach number weight because none of the above steps creates negative power of Mach number. Hence, we can conclude that
			\begin{align}
				\sum_{k=0}^{2}\ddt \left(\left\|(\eps\p_t)^k\nab\times  (\rho_0 u)\right\|_{2-k}^2 +\sum_{j=1}^{3}\left\|(\eps\p_t)^k\nab\times (\rho_0 F_j)\right\|_{2-k}^2\right) \leq P(E(t)).
			\end{align}
		\end{proof}

		We proceed to derive the estimates of the curl parts $\nabla\times u$ and $\nabla\times F_j$.
		\begin{cor}\label{lem_Esti_curl}
			For $k=0,1,2$, under the assumptions of Theorem \ref{main thm, ill data}, we have
			\begin{align}
				\|(\eps\p_t)^k\nabla\times u\|_{2-k}^2  \le& \left( 1 + \sum_{l=0}^{k}\|(\eps\p_t)^l u\|_{2-k}^2  \right) P(E_1(t)),\\
				\|(\eps\p_t)^k\nabla\times F_j\|_{2-k}^2  \le& \left( 1+ \sum_{l=0}^{k}\|(\eps\p_t)^lF_j\|_{2-k}^2  \right) P(E_1(t)).
			\end{align}
		\end{cor}
		\begin{proof}
			Since $\nabla\times u= \rho_0^{-1}( \nabla\times(\rho_0 u) -\nabla\rho_0 \times u )$, the curl part of $u$ controlled by
			\begin{align}
				\|(\eps\p_t)^k\nabla\times u\|_{2-k}^2  \le& ~\|(\eps\p_t)^k(\rho_0^{-1}\nabla\times(\rho_0 u))\|_{2-k}^2  + \|(\eps\p_t)^k(\rho_0^{-1}\nabla\rho_0\times u)\|_{2-k}^2  \notag\\
				\lesssim & ~\|S\|_{2,\eps}^2  \|\nabla\times(\rho_0 u)\|_{2,\eps}^2  + \left(\sum_{l=0}^{k}\|(\eps\p_t)^l u\|_{2-k}^2\right) P(\|S\|_{3,\eps})\notag\\
				\le& ~\left( 1+ \sum_{l=0}^{k}\|(\eps\p_t)^l u\|_{2-k}^2 \right)P(E_1(t)).
			\end{align}
			Similarly, we obtain
			\begin{equation}
				\|(\eps\p_t)^k\nabla\times F_j\|_{2-k}^2  \le \left( 1+ \sum_{l=0}^{k}\|(\eps\p_t)^lF_j\|_{2-k}^2  \right) P(E_1(t)).
			\end{equation}
			
		\end{proof}
		
		\begin{rmk}\label{rmk_Esti_curl}
			It should be noted that the conclusion of this corollary is not the end of the reduction, as there are still normal derivatives in $\|(\eps\p_t)^l (u, F_j)\|_{2-k}^2$ for $0\leq k\leq 2,~0\leq l\leq k$. But these terms are lower order terms and can be reduced to the control of divergence and the full time derivatives $\|(\eps\p_t)^k (u, F_j)\|_{0}^2,~0\leq k\leq 3$ (which is a part of $E_1(t)$) by repeatedly applying the div-curl decomposition.
		\end{rmk}
		
		\subsection{Control of divergence and reduction of pressure}\label{sect div}
		Next, we are going to derive the estimate of $\nabla q$, $\nabla\cdot u$, $\nabla\cdot F_j$ as well as their time derivatives. 
		\begin{prop}\label{lem_Esti_div}
			For $k=0,1,2$, under the assumptions of Theorem \ref{main thm, ill data}, we have for any $\delta\in(0,1)$ that
			\begin{align}
				\|(\eps\p_t)^k\nabla q\|_{2-k}^2 + \|(\eps\p_t)^k\nabla \cdot u\|_{2-k}^2  + \sum_{j=1}^3\|(\eps\p_t)^k\nabla \cdot F_j\|_{2-k}^2  \leq \delta E(t) + P(E(0))+ P(E(t))\int_0^t P(E(\tau))\dtau.
			\end{align}
		\end{prop}
		We get from the continuity equation and the divergence constraint of $F_j$ in \eqref{CElasto3} that
		\begin{align}
			\label{eq nabq}	-\nabla q=&~ \eps \rho D_t u -\eps \rho \sum_{j=1}^{3}(F_j +\BF_j)\cdot\nabla F_j,\\
			\label{eq divu}	-\nabla\cdot u=& ~\eps a D_t q ,\\
			\label{eq divF}	-\nabla\cdot F_j=& ~\rho^{-1} (F_j+\BF_j)\cdot\nabla \rho =\rho^{-1}\left[ \eps\p_q \rho(F_j+\BF_j)\cdot\nabla q + \p_S \rho(F_j+\BF_j)\cdot\nabla S \right] \notag\\
			\overset{a=\rho^{-1}\p_q \rho}{===}&~\eps a \fjp q + b(F_j+\BF_j)\cdot\nabla S 
		\end{align}where $b=b(\eps q, S):=\frac{1}{\rho}\frac{\p \rho}{\p S}$ is a smooth function in its arguments.
		
		The control of  $\| b(F_j+\BF_j)\cdot\nabla S \|_{2,\eps}$ is straightforward
		\[
		\ddt\| b(F_j+\BF_j)\cdot\nabla S \|_{2,\eps}^2\leq P(E(t)).
		\] thanks to Corollary \ref{lem_EnergyEsti_S}. For the terms $\eps D_t q,~\nab q$ and $\eps \fjp q$, their $\|\cdot\|_{2,\eps}^2$ norms can be controlled by $P(E_1(t))+C\eps E(t)$, and the smallness of $\eps>0$ is used to absorb $C\eps E(t)$ when closing the energy estimates, as shown in previous works \cite{Schochet1986limit, Metivier2001limit, Alazard2005limit} about the low Mach number limit of non-isentropic Euler equations. In this paper, we would like to drop the dependence on the smallness of $\eps$. 
		
		It should be noted that $\eps D_t =\eps \p_t +\eps u\cdot\nab$, so we can alternatively try to obtain the bounds for $$\|(\eps D_t)^k\nabla q\|_{2-k}^2, \|(\eps D_t)^k\nabla \cdot u\|_{2-k}^2, \|(\eps D_t)^k\nabla \cdot F_j\|_{2-k}^2, $$ which is more technically convenient in the analysis of pressure and divergence. 
		\begin{prop}\label{lem_Esti_div_1}
			For $k=0,1,2$, under the assumptions of Theorem \ref{main thm, ill data}, we have for any $\delta\in(0,1)$ that
			\begin{align}
				\|(\eps D_t)^k\nabla q\|_{2-k}^2 + \|(\eps D_t)^k\nabla \cdot u\|_{2-k}^2  + \sum_{j=1}^3 \|(\eps D_t)^k\nabla \cdot F_j\|_{2-k}^2  \leq \delta E(t) + P(E(0))+ P(E(t))\int_0^t P(E(\tau))\dtau.
			\end{align}
		\end{prop}
		
		In fact, we can prove the conclusion of Proposition \ref{lem_Esti_div_1} implies the conclusion of Proposition \ref{lem_Esti_div}.
		\begin{proof}[Proof of ``Prop. \ref{lem_Esti_div_1} $\Rightarrow$ Prop. \ref{lem_Esti_div}"]
			It suffices to consider the case $k=1,2$. For $k=1$, we have $\eps D_t (\nab q)-\eps \p_t (\nab q)=\eps (u\cdot\nab) \nab q$. So, we have
			\begin{align*}
				&\|\eps D_t(\nab q)-\eps \p_t(\nab q)\|_1^2=\|\eps (u\cdot\nab) \nab q\|_1^2\lesssim \|\eps u\|_2^2\|q\|_3^2.
			\end{align*}Using the result for $k=0$ (remember that we are assuming the conclusion of Proposition \ref{lem_Esti_div_1} holds at this step), we get
			\[
			\|q\|_3^2\leq \delta E(t) + P(E(0))+ P(E(t))\int_0^t P(E(\tau))\dtau.
			\] For $\|\eps u\|_2^2=\io |\eps u|^2 + |\eps\p u|^2+ |\eps \p^2 u|^2\dx$, using AM-GM inequality and Jensen's inequality, we get for $j=0,1,2$ that
			\begin{align*}
				\io |\eps\p^j u|^2\dx=&\io\left(\eps\p^j u_0 + \int_0^t \eps\p^j \p_t u (\tau,\cdot)\dtau\right)^2\dx\lesssim \io |\eps\p^j u_0|^2\dx + \io\left( \int_0^t \eps \p_t\p^j u (\tau,\cdot)\dtau\right)^2\dx\\
				\lesssim&\io |\eps\p^j u_0|^2\dx+\io \int_0^t |\eps \p_t\p^j u (\tau,\cdot)|^2\dtau\dx\\
				=&\io |\eps\p^j u_0|^2\dx+\int_0^t\io  |\eps \p_t\p^j u (\tau,\cdot)|^2\dx\dtau,
			\end{align*}and thus
			\[
			\|\eps u\|_2^2\lesssim \eps^2\|u_0\|_2^2+\int_0^t \|\eps\p_t u(\tau,\cdot)\|_2^2 \dtau \leq \eps^2 E(0) +\int_0^t E(t)\dt
			\]
			Combining these inequalities, we get
			\begin{align*}
				\|\eps D_t(\nab q)-\eps \p_t(\nab q)\|_1^2\lesssim~ \delta E(t) + P(E(0))+ P(E(t))\int_0^t P(E(\tau))\dtau.
			\end{align*}Similarly, we can prove that
			\[
			\|\eps D_t(\nab \cdot u)-\eps \p_t(\nab \cdot u)\|_1^2+ \sum_{j=1}^3\|\eps D_t(\nab \cdot F_j)-\eps \p_t(\nab \cdot F_j)\|_1^2\lesssim \delta E(t) + P(E(0))+ P(E(t))\int_0^t P(E(\tau))\dtau.
			\]
			When $k=2$, we have 
			\begin{align*}
				\eps^2 D_t^2\nab q -\eps^2\p_t^2 \nab q = \eps^2 (\p_t u_i)(\p_i q) + 2\eps^2 u_i\p_i\p_t q + \eps^2 u_i(\p_i u_j)(\p_j  q) +\eps^2 u_iu_j\p_i\p_j q.
			\end{align*}Using Corollary \ref{cor product} and Jensen's inequality as above, we have
			\begin{align*}
				&\|\eps^2 u_i\p_i\p_t q \|_0^2\leq \delta\|\eps \nab\p_t q\|_1^2+\|\eps u\|_1^4\|\eps \nab\p_t q\|_0^2\\
				\lesssim&~\delta E(t) + \|\eps \nab\p_t q\|_0^2\left(\eps^4\|u_0\|_1^4+\int_0^t\|\eps \p_t u(\tau,\cdot)\|_1^4\dtau\right)\\
				\lesssim&~\delta E(t) + E(t)\int_0^t P(E(\tau))\dtau +\eps^2(E(0))^2\|\eps^2\p_t q\|_1^2\\
				\lesssim&~\delta E(t) + E(t)\int_0^t P(E(\tau))\dtau +\eps^2(E(0))^2\left(\|\eps^2\p_t q(0,\cdot)\|_1^2+\int_0^t\|(\eps \p_t)^2 q(\tau,\cdot)\|_1^2\dtau\right)\\
				\lesssim&~\delta E(t) + P(E(0))+  (1+E(t))\int_0^t P(E(\tau))\dtau.
			\end{align*}The other terms can be analyzed in the same manner and we do not repeat the details here.
		\end{proof}
		
		In the rest of this section, we are devoted to proving Proposition \ref{lem_Esti_div_1}.
		
		\subsubsection{Derivation of the wave equation of $q$}
		For compressible inviscid fluids, the pressure $q$ satisfies a wave-type equation and we now derive the concrete form of the wave equation. Taking $D_t$ in the continuity equation $\eps a D_t q + \nab \cdot u=0$, inserting the concrete form of $[\p,D_t]$ and using $D_t S=0$, we get
		\begin{align}\label{wave eq 0}
			\eps a D_t^2 q + \nab\cdot D_t u = -(\eps D_t q) D_t a +\p_i u_j\,\p_j u_i = \p_i u_j\, \p_j u_i - \frac{\p a}{\p q}(\eps D_t q)^2.
		\end{align}Inserting the momentum equation $D_t u = -(\eps\rho)^{-1}\nab q +\sum_j \fjp F_j$, the term $ \nab\cdot D_t u$ becomes
		\begin{align*}
			\nab\cdot D_t u =& -\eps^{-1}\nab\cdot(\rho^{-1}\nab q) + \sum_{j=1}^3\nab\cdot(\fjp F_j) \\
			=& -\eps^{-1}\nab\cdot(\rho^{-1}\nab q) + \sum_{j=1}^3 \fjp(\nab\cdot F_j) + \p_i F_{kj}\, \p_k F_{ij}
		\end{align*}Inserting \eqref{eq divF}, we obtain
		\begin{align*}
			&\fjp(\nab\cdot F_j)\notag\\
			=&-\eps a (\fjp)^2 q - b (\fjp)^2S \\
			&-\eps^2\p_q a|\fjp q|^2- \eps (\p_S a + \p_q b)(\fjp q) (\fjp S) - \p_S b |\fjp S|^2.
		\end{align*}
		Plugging the expressions of $\nab\cdot D_t u$ back to \eqref{wave eq 0}, we obtain the wave-type equation of $q$ with Neumann boundary condition (obtained by restricting the third component of momentum equation onto $\Sigma$)
		\begin{equation}\label{wave eq q1}
			\begin{cases}
				\eps^2 a D_t^2 q - \nab \cdot (\rho^{-1}\nab q) - \sum\limits_{j=1}^3\eps^2 a (\fjp)^2 q = \mathcal{G}^\eps&\text{ in }\Om,\\
				\p_3 q=0&\text{ on }\Sigma,
			\end{cases}
		\end{equation}where the source term $\mathcal{G}^\eps$ consists of the following terms
		\begin{align}
			\label{wave source 1}\mathcal{G}^\eps:=&\sum\limits_{j=1}^3\eps \left(b (\fjp)^2S + \p_S b |\fjp S|^2 + \p_i u_j\, \p_j u_i - \p_i F_{kj}\, \p_k F_{ij} \right)\\
			&+\eps^2 (\p_S a + \p_q b)(\fjp q) (\fjp S) +\eps^3\p_q a \left(|\fjp q|^2-(D_t q)^2\right).\notag
		\end{align}This formulation of wave equation will be used to establish the wave-type estimates, as it is straightforward to see that the source term $\mathcal{G}^\eps$ satisfies the following uniform bound
		\begin{align}\label{uniform G}
			\|\eps^{-1}\mathcal{G}^\eps\|_{2,\eps}\lesssim  P(E(t)).
		\end{align}  In particular, we immediately obtain uniform $L^2(\Om)$ estimates for $\eps D_t q, \nab q$ and $\eps\fjp q$.
		\begin{lem}[Uniform $L^2(\Om)$ estimate of the wave equation]\label{lem wave L2}
			Define 
			\begin{align}\label{energy W0}
				\W_0(t):=\frac12\io a |\eps D_t q|^2 + \rho^{-1} |\nab q|^2 + \sum_{j=1}^3 a |\eps\fjp q|^2\dx.
			\end{align} Then it satisfies $\dfrac{\mathrm{d}\W_0(t)}{\dt}\leq P(E(t))$.
		\end{lem}
		\begin{proof}
			Invoking \eqref{wave eq q1} and integrating by parts, we get
			\begin{align*}
				&\ddt\frac12\io a|\eps D_t q|^2\dx = \io a\eps^2 D_t^2 q\, D_t q\dx +\frac12\io \left(D_t a +(\nab\cdot u)a\right)|\eps D_t q|^2\dx\\
				=&\io\nab\cdot(\rho^{-1}\nab q)\,D_t q\dx +\sum\limits_{j=1}^3 \io \eps^2 a (\fjp)^2 q\, D_t q\dx \\
				&+\io \mathcal{G}^\eps D_t q + \frac12\left(D_t a +(\nab\cdot u)a\right)|\eps D_t q|^2\dx\\
				=&-\io \rho^{-1}\nab q\cdot D_t \nab q + \sum\limits_{j=1}^3 a \left(\fjp q\right)\, \left(D_t \fjp q \right)\dx\\
				&+\io \sum\limits_{j=1}^3\eps^2 (\fjp a) (\fjp q)D_tq +\rho^{-1}\nab q\cdot [D_t,\nab] q + \mathcal{G}^\eps D_t q + \frac12\left(D_t a +(\nab\cdot u)a\right)|\eps D_t q|^2\dx,
			\end{align*}where we also use the fact $[D_t,\fjp]=0$. Since $\|\mathcal{G}^\eps\|_{2,\eps}\lesssim \eps P(E(t))$ holds uniformly in $\eps$, we know the last line is uniformly bounded
			\[
			\io\sum\limits_{j=1}^3 \eps^2 (\fjp a) (\fjp q)D_tq +\rho^{-1}\nab q\cdot [D_t,\nab] q + \mathcal{G}^\eps D_t q + \frac12\left(D_t a +(\nab\cdot u)a\right)|\eps D_t q|^2\dx\leq P(E(t)).
			\] Thus, we obtain 
			\begin{align*}
				&-\io \rho^{-1}\nab q\cdot D_t \nab q +\sum\limits_{j=1}^3 a \left(\fjp q\right)\, \left(D_t \fjp q \right)\dx\\
				=&-\frac12\ddt\io \rho^{-1}|\nab q|^2 +\sum\limits_{j=1}^3 a\left|\eps\fjp q\right|^2 \dx\\
				&-\frac12\io\rho^{-2}(D_t\rho -\rho(\nab\cdot u))|\nab q|^2 - \sum\limits_{j=1}^3 \left(D_t a +(\nab\cdot u)a\right)|\eps\fjp q|^2\dx\\
				\lesssim&-\frac12\ddt\io \rho^{-1}|\nab q|^2 + \sum\limits_{j=1}^3 a\left|\eps\fjp q\right|^2 \dx+P(E(t)).
			\end{align*}which leads to our desired $L^2(\Om)$ estimates.
		\end{proof}
		
		One may also alternatively write $\nab\cdot(\rho^{-1}q)=\nab( \rho^{-1} )\cdot \nab q + \rho^{-1}\lap q$ to obtain another formulation of the wave equation
		\begin{equation}\label{wave eq q2}
			\begin{cases}
				\eps^2 \rho a D_t^2 q - \lap q - \sum\limits_{j=1}^3 \eps^2 \rho a (\fjp)^2 q = \wg^\eps&\text{ in }\Om,\\
				\p_3 q=0&\text{ on }\Sigma,
			\end{cases}
		\end{equation}with
		\begin{align}\label{wave source 2}
			\wg^\eps=\rho \mathcal{G}^\eps - \rho^{-1}\nab\rho\cdot\nab q = \rho \mathcal{G}^\eps - \eps a |\nab q|^2 - b \nab S\cdot\nab q.
		\end{align} 
		\begin{rmk}Do note that the source term $\wg^\eps$ now contains an $O(1)$ size term. This formulation is not suitable to prove uniform-in-$\eps$ wave-type estimates, especially when we differentiate \eqref{wave eq q2} by time derivatives. Instead, \eqref{wave eq q2} will be used to reduce the normal derivatives falling on $\nab q$, $\eps\fjp q$ and $\eps D_t q$.
		\end{rmk}
		
		\subsubsection{Reduction of pressure via elliptic estimates and wave equation}
		Now, we let $X=\nab q$ and $s=2$ in Lemma \ref{lem divcurl} to get
		\begin{align}\label{reduce nabq 1}
			\|\nab q\|_2^2 \lesssim \|\nab q\|_0^2 + \|\lap q\|_1^2 + \underbrace{\|\nab\times\nab q\|_1^2 + |\p_3 q|_{1.5}^2}_{=0} =  \|\nab q\|_0^2 + \|\lap q\|_1^2.
		\end{align}
		The term $\|\nab q\|_0^2$ has been controlled in Lemma \ref{lem wave L2}. For $\|\lap q\|_1^2$, we insert the wave equation \eqref{wave eq q2} to get
		\begin{align}\label{reduce nabq 2}
			\|\lap q\|_1^2\lesssim&~\|\eps^2 D_t^2 q\|_1^2 + \sum\limits_{j=1}^3 \|\eps^2 (\fjp)^2 q\|_1^2 + \|\wg^\eps\|_1^2 \notag\\
			&+ \left(\|\eps^2 D_t^2 q\|_0^2  + \sum\limits_{j=1}^3 \|\eps^2 (\fjp)^2 q\|_0^2\right)\|\rho a\|_{W^{1,\infty}}^2.
		\end{align}
		Thus, we shall seek for the uniform-in-$\eps$ control for $\|\eps^2 D_t^2 q\|_1^2$ and $\sum\limits_{j=1}^3\|\eps^2 (\fjp)^2 q\|_1^2$. The control of the remainder terms, namely $ \|\wg^\eps\|_1^2 $ and the second line of \eqref{reduce nabq 2}, is also postponed to later sections.
		
		In view of \eqref{eq divu}-\eqref{eq divF}, we also need to control $\|\eps D_t q\|_2^2$ and $\|\eps \flp q\|_2^2$ for $l=1,2,3$. We have
		\begin{align}
			\|\eps D_t q\|_2^2+\|\eps \flp q\|_2^2\lesssim \|\eps D_t q\|_0^2 +\|\eps \flp q\|_0^2+ \|\eps\nab D_t q\|_1^2+\|\eps \nab(\flp q)\|_1^2.
		\end{align} We let $s=1$ and $X=\nab D_t q$ and $\flp q~(l=1,2,3)$ respectively in Lemma \ref{lem divcurl} and $D_t\p_3 q=\flp\p_3 q=0$ on $\Sigma$ to get
		\begin{align}
			\|\eps\nab D_t q\|_1^2\lesssim&~ \|\eps\nab D_t q\|_0^2 + \|\eps\lap D_t q\|_0^2 + |\eps \p_3 D_t q|_{\frac12}^2\notag\\
			\label{reduce Dtq 1} \lesssim &~\|\eps D_t\lap q\|_0^2+\|\eps\nab D_t q\|_0^2 +\|\eps[\lap, D_t] q\|_0^2 + |\eps [\p_3, D_t]q|_{\frac12}^2\\
			\|\eps \nab(\flp q)\|_1^2\lesssim&~ \|\eps \nab(\flp q)\|_0^2 + \|\eps \lap \flp q\|_0^2+ |\eps \p_3 \flp q|_{\frac12}^2 \notag \\
			\lesssim&~\|\eps  \flp\lap q\|_0^2+\|\eps \nab(\flp q)\|_0^2\notag\\
			& +\|\eps[\lap, \flp] q\|_0^2+ |\eps [\p_3, \flp]q|_{\frac12}^2.\label{reduce FPq 1}
		\end{align}
		
		We then focus on the reduction of major terms $\|\eps D_t\lap q\|_0^2$ and $\|\eps  \flp\lap q\|_0^2$ and postpone the control of numerous lower-order remainder terms to later sections. Again, we invoke the wave equation \eqref{wave eq q2} to get
		\begin{align}\label{reduce Dtq 2}
			\eps D_t\lap  q=\eps^3 \rho a D_t^3 q - \eps^3 \rho a (\fjp)^2 D_t q - \eps D_t\wg^\eps+\cdots
		\end{align}and
		\begin{align}\label{reduce FPq 2}
			\eps \flp\lap  q=&~\eps^3 \rho a D_t^2\flp q - \sum_{j=1}^3 \eps^3 \rho a (\fjp)^2\flp q \\
			& - \eps \flp\wg^\eps+\cdots\notag
		\end{align}where the omitted terms are those generated when $D_t$ or $\flp$ falls on the coefficients $a$ and $\rho$. These omitted terms are lower-order and have no loss of $\eps$ weights because $a,\rho$ are functions of $\eps q$ and S and we have $D_t S=0$ and enhanced regularity for $\flp S$.
		
		To prove Proposition \ref{lem_Esti_div_1}, we also need to control the following terms
		\begin{equation}\label{reduce div 3}
			\begin{aligned}
				\|\eps^2 D_t^2 \nab q\|_0^2\lesssim&~ \|\eps^2\nab D_t^2  q\|_0^2+\|\eps^2 [D_t^2, \nab] q\|_0^2,\\
				\|\eps^2 D_t\flp \nab q\|_0^2\lesssim&~ \|\eps^2\nab D_t\flp  q\|_0^2+\|\eps^2 [D_t\flp, \nab] q\|_0^2,\\
				\|\eps^3 D_t^3q\|_0^2&\text{ and }\|\eps^3 D_t^2\flp q\|_0^2.
			\end{aligned}
		\end{equation}
		In summary, we shall seek for uniform-in-$\eps$ control of 
		\begin{equation}\label{reduce div}
			\begin{aligned}
				&\|\eps^3 D_t^3 q\|_0^2,~\|\eps^2 \nab D_t^2 q\|_0^2~~~(l=1,2,3),\\
				&\|\eps^2\nab D_t\flp  q\|_0^2,~\|\eps^3D_t^2\flp q\|_0^2,\\
				&\sum_{j=0}^3\|\eps^3 D_t (\fjp)^2 q\|_0^2,~\sum_{j=0}^3\|\eps^3  (\fjp)^2\flp q\|_0^2~~~(l=1,2,3).
			\end{aligned}
		\end{equation} and also the control of those remainder terms and commutators in \eqref{reduce Dtq 1}-\eqref{reduce div 3}.
		
		\subsubsection{Uniform estimates of tangentially-differentiated wave equations}
		We define
		\begin{align}
			\W_{1}^\eps(t):=\frac12\io& a\left|\eps^3D_t^3 q\right|^2 + \rho^{-1} \left|\eps^2\nab D_t^2 q\right|^2 + \sum_{j=1}^3a \left|\eps^3 D_t^2 \fjp q\right|^2\dx,\\
			\W_{2,l}^\eps(t):=\frac12\io& a\left|\eps^3D_t^2 \flp q\right|^2 + \rho^{-1} \left|\eps^2\nab\left( D_t \flp q\right)\right|^2 \\
			&+ \sum_{j=1}^3a \left|\eps^3 D_t \flp \left(\fjp q\right)\right|^2\dx~~~(l=1,2,3), \notag\\
			\W_{3,l}^\eps(t):=\frac12\io& a \left|\eps^3D_t (\flp)^2 q\right|^2 + \rho^{-1} \left|\eps^2\nab\left((\flp)^2 q\right)\right|^2 \\
			&+ \sum_{j=1}^3a \left|\eps^3(\flp)^{2}\left(\fjp q\right)\right|^2 \dx~~~(l=1,2,3),\notag
		\end{align} which are exactly the energy functionals for $D_t^2$-differentiated, $D_t\flp$-differentiated and $(\flp)^2$-differentiated wave equation \eqref{wave eq q1}. In this section, we prove the following conclusion.
		\begin{lem}[Uniform estimates for tangentially-differentiated wave equations]\label{lem wave H3}
			The energy functionals  $\W_{1}^\eps(t)$, $\W_{2,l}^\eps(t)$, $\W_{3,l}^\eps(t)$ obey the following uniform-in-$\eps$ estimates for any $\delta\in(0,1)$
			\begin{align}
				\W_{1}^\eps(t)+\sum_{l=1}^3\W_{2,l}^\eps(t)+\W_{3,l}^\eps(t)\leq \delta E(t) + P(E(0))+ P(E(t))\int_0^t P(E(\tau))\dtau.
			\end{align}
		\end{lem}
		Once this lemma is proven, we immediately obtain the control for the terms in \eqref{reduce div}.
		
		\begin{proof}
			We only prove the estimate for $\W^\eps_1(t)$, that is, the uniform energy estimates for $D_t^2$-differentiated wave equation \eqref{wave eq q1}. This is the most difficult cases because $D_t^2$ involves more time derivatives than $D_t\flp$ and $(\flp)^2$ and then requires more $\eps$ weights to prove the uniform estimates. The other two cases can be proved in exactly the same manner if we recall the concrete forms of those commutators recorded in Lemma \ref{commutators}. 
			
			Taking $D_t^2$ and multiplying $\eps$ in \eqref{wave eq q1}, we obtain
			\begin{align}\label{wave eq qttt1}
				\eps^3 a D_t^4 q - \eps\nab\cdot(\rho^{-1} D_t^2 \nab q) - \eps^3 a (\fjp)^2D_t^2 q = G_{3,2}^\eps
			\end{align}
			where
			\begin{align*}
				G_{3,2}^\eps:=&~\eps^2 D_t^2(\eps^{-1}\mathcal{G}^\eps) + \eps^3 D_t^2 a(D_t^2 q - (\fjp)^2 q) +2\eps^3 D_t a(D_t^3 q - (\fjp)^2 D_t q)\\
				&+\eps [D_t^2,\nab\cdot](\rho^{-1}\nab q)+\eps \nab\cdot([D_t^2,\rho^{-1}]\nab q).
			\end{align*}
			Testing \eqref{wave eq qttt1} with $\eps^3 D_t^3 q$ in $L^2$, we obtain 
			\begin{align*}
				&\ddt\frac12\io a |\eps^3D_t^3 q|^2 \dx= \io a\eps^6 D_t^4 q\,D_t^3 q\dx +\frac12\io (D_t a+ (u\cdot \nab) a)|\eps^3 D_t^3 q|^2\dx\\
				=&\io \eps\nab\cdot(\rho^{-1}D_t^2\nab  q)\, (\eps^3 D_t^3 q)\dx +\io \eps^3 a (\fjp)^2 D_t^2 q \, (\eps^3 D_t^3 q)\dx\\
				+& \io  G_{3,2}^\eps\,(\eps^3 D_t^3 q)\dx+\frac12\io (D_t a+ (u\cdot \nab) a)|\eps^3 D_t^3 q|^2\dx
			\end{align*}
			Integrating by parts and using $D_t^2 \p_3 q=F_{3j}+\BF_{3j}=0$ on $\Sigma$, we get
			\begin{align*}
				&\io \eps\nab\cdot(\rho^{-1}D_t^2\nab  q)\, (\eps^3 D_t^3 q)\dx + \sum_{j=1}^3\io \eps^3 a (\fjp)^2 D_t^2 q \, (\eps^3 D_t^3 q)\dx\\
				=&-\io\rho^{-1}(\eps^2 \nab D_t^2 q)\,D_t (\eps^2 \nab D_t^2 q) \dx - \sum_{j=1}^3 \io a (\eps^3 \fjp D_t^2 q)\, D_t(\eps^3 \fjp D_t^2 q)\dx\\
				&+\io\rho^{-1} (\eps^2 D_t^2 \nab q)\,[D_t, \nab](\eps^2  D_t^2 q) - \sum_{j=1}^3 \eps^3\nab\cdot(a(F_j+\BF_j))\,D_t^2 q\,(\eps^3D_t^3 q)\dx\\
				&\underbrace{-\io\eps^4\rho^{-1}[D_t^2,\nab]q\cdot D_t(\nab D_t^2 q)\dx}_{\mathcal{J}}\\
				=&-\frac12\ddt\io \rho^{-1}\left|\eps\nab D_t^2 q\right|^2 +\sum_{j=1}^3 a\left|\eps^3\fjp D_t^2 q\right|^2\dx\\
				&-\sum_{j=1}^3\frac12\io \rho^{-2}(D_t\rho -\rho(\nab\cdot u)) - (D_t a+(u\cdot \nab)a)\left|\eps^3\fjp D_t^2 q\right|^2\dx\\
				&+\io\rho^{-1} (\eps^2 \nab D_t^2 q)\,[D_t, \nab](\eps^2  D_t^2 q) - \sum_{j=1}^3 \eps^3\nab\cdot(a(F_j+\BF_j))\,D_t^2 q\,(\eps^3D_t^3 q)\dx+\mathcal{J}\\
				\lesssim&-\frac12\ddt\io \rho^{-1}\left|\eps\nab D_t^2 q\right|^2 +\sum_{j=1}^3 a\left|\eps^3\fjp D_t^2 q\right|^2\dx+P(E(t))+\mathcal{J}.
			\end{align*}
			It remains to control $\mathcal{J}$, in which we should integrate by parts $D_t$ to avoid loss of derivative:
			\begin{align*}
				\int_0^t\mathcal{J}(\tau,\cdot)\dtau\overset{D_t}{==}&\int_0^t\io\rho^{-1}\eps^2 D_t\left([D_t^2,\nab]q\right)\cdot (\eps^2\nab D_t^2 q) \dx\dtau \\
				&+\int_0^t\io\left(D_t\rho-\rho(\nab\cdot u)\right)\,  \left(\eps^2[D_t^2,\nab]q\right)\cdot (\eps^2\nab D_t^2 q) \dx\dtau\\
				&+\io \rho^{-1}\eps^2\left([D_t^2,\nab]q\right)\cdot (\eps^2\nab D_t^2 q) \dx\bigg|^t_0,
			\end{align*}where the first two terms are bounded by $\int_0^tP(E(\tau))\dtau$ by direct computation because the commutator $[D_t^2,\nab]q$ only consists of terms in the form of $(\p D_tu)(\p q),~(\p D_t q)(\p u)$ or $(\p u)(\p u)(\p q)$ according to Lemma \ref{commutators}. As for the last term, we again invoke the concrete form of the commutator to see the highest-order part has the form
			\begin{align*}
				\mathcal{J}_0:=\io \rho^{-1}\eps^2(\p D_t X)(\p Y) (\eps^2\nab D_t^2 q) \dx
			\end{align*}where $(X,Y)=(q,u)$ or $(u,q)$. Using Young's inequality and Sobolev interpolation, we know
			\begin{align*}
				\mathcal{J}_0\lesssim&~ \delta\|\eps^2\nab D_t^2 q\|_0^2 + \|\eps\rho^{-1}(\p Y)\|_{L^6}^2\|\eps \p D_t X\|_{L^3}^2\lesssim\delta\|\eps^2\nab D_t^2 q\|_0^2 + \|\eps\rho^{-1}(\p Y)\|_{1}^2\|\eps \p D_t X\|_{\frac12}^2\\
				\lesssim&~\delta(\|\eps^2\nab D_t^2 q\|_0^2 +\|\eps \p D_t X\|_{1}^2) + \|\eps\rho^{-1}(\p Y)\|_{1}^4\|\eps \p D_t X\|_{0}^2\\
				\lesssim&~\delta E(t) + \|\eps \rho^{-1}(\p Y)\|_1^4\|\eps\p D_t X\|_0^2.
			\end{align*}It is easy to see
			\[
			\|\eps \rho^{-1}(\p Y)(t)\|_1^4\lesssim \|\eps \rho^{-1}(\p Y)(0,\cdot)\|_1^4 + \int_0^t \|\eps\p_t (\rho^{-1}(\p Y))(\tau,\cdot)\|_0^4\dtau\leq \eps^4 \|\rho^{-1}(\p Y)(0,\cdot)\|_1^8 +\int_0^t P(E(\tau))\dtau,
			\]which then leads to 
			\begin{align*}
				\|\eps \rho^{-1}(\p Y)\|_1^4\|\eps \p D_t X\|_0^2&\lesssim \eps^2\|\rho^{-1}(\p Y)(0,\cdot)\|_1^4\|\eps^2\p D_t X\|_0^2 +\|\eps\p D_t X\|_0^2\int_0^t P(E(\tau))\dtau\\
				\lesssim&~\eps^2\|\rho^{-1}(\p Y)(0,\cdot)\|_1^4\left(\|\eps^2 \p D_t X(0,\cdot)\|_0^2+\int_0^t \|\eps^2 \p_t\p D_t X(\tau,\cdot)\|_0^2\dtau\right) + P(E(t))\int_0^t P(E(\tau))\dtau\\
				\lesssim &~P(E(0))+P(E(t))\int_0^t P(E(\tau))\dtau.
			\end{align*}and thus
			\begin{align}
				\int_0^t \mathcal{J}(\tau,\cdot)\dtau\leq \delta E(t) + P(E(0))+ P(E(t))\int_0^t P(E(\tau))\dtau.
			\end{align}
			Using the uniform bound $\|\eps^{-1}\mathcal{G}^{\eps}\|_{2,\eps}\leq P(E(t))$ and the concrete forms of the commutators recorded in  Lemma \ref{commutators}, we obtain $\|G_{3,2}^\eps\|_{0}\leq P(E(t))$ and thus the above analysis give us 
			\[
			\W_{1}^\eps(t) \leq \delta E(t) + P(E(0))+ P(E(t))\int_0^t P(E(\tau))\dtau.
			\]as desired. For the other two cases, we can replace $D_t^2$ by $D_t^k(\flp)^{2-k}$ for $k=0,1$ and obtain the desired energy bound in the same way.
		\end{proof}
		\begin{rmk}
			The appearance of $\mathcal{J}$ is necessary. In fact, when integrating by parts, we must eliminate all boundary terms by differentiating the boundary condition $\p_3 q|_{\Sigma}=0$ with $D_t^k(\fjp)^{2-k}$. Such derivatives have variable coefficients and we cannot commute them with $\p_3$, that is, we may not ensure $\p_3\left(D_t^k(\fjp)^{2-k}q\right)=0$ on $\Sigma$.
		\end{rmk}

		\subsubsection{Uniform estimates of remainder terms}
		We already control the top-order terms via the tangentially-differentiated wave equations. Applying the same strategy to $\eps D_t$-differentiated and $\fjp$-differentiated wave equation \eqref{wave eq q1}, we can obtain the same control for the terms in the second line of \eqref{reduce nabq 2}. Now, we turn to control to remainder terms that are generated in the reduction process. We aim to prove the following uniform-in-$\eps$ estimates for all these remainder terms.
		\begin{lem}[Remainder estimates]\label{lem wave remainders}
			For all $\delta\in(0,1)$, there hold the following uniform-in-$\eps$ estimates
			\begin{align}
				&\|\nab q\|_0^2+\|\eps\nab D_t q\|_0^2+\sum_{j=1}^3\|\eps\nab(\fjp q)\|_0^2\lesssim P(E(0))+\int_0^t P(E(\tau))\dtau.\\
				&\|\eps[\lap, D_t]q\|_0^2+|\eps[\p_3,D_t]q|_{\frac12}^2 + \|\eps^2[D_t^2,\nab] q\|_0^2\notag\\
				+&\sum_{l=1}^3\|\eps[\lap,\flp]q\|_0^2+|\eps[\p_3,\flp]q|_{\frac12}^2+ \|\eps^2[D_t\flp,\nab] q\|_0^2\notag\\
				\lesssim&~ \delta E(t)+ P(E(0))+P(E(t))\int_0^t P(E(\tau))\dtau\\
				&\|\wg^\eps\|_1^2+\|\eps D_t \wg^\eps\|_0^2+\sum_{l=1}^3\|\eps\flp\wg^\eps\|_0^2\lesssim~ \delta E(t) + P(E(0))+P(E(t))\int_0^t P(E(\tau))\dtau.
			\end{align}
		\end{lem}
		\begin{proof} The remainder terms are classified into three types as shown in the lemma.
			\paragraph*{The $L^2$ terms in the div-curl inequality.} In \eqref{reduce nabq 1}, \eqref{reduce Dtq 1} and \eqref{reduce FPq 1}, we shall also control the terms
			\[
			\|\nab q\|_0^2,\q\|\eps\nab D_t q\|_0^2,\q \|\eps\nab(\flp q)\|_0^2.
			\]The first term has been controlled in Lemma \ref{lem wave L2} (the $L^2(\Om)$ estimate of the wave equation). The second term is part of the energy functional for $\eps D_t$-differentiated wave equations \eqref{wave eq q1} and the third term is part of the energy functional for $\eps \fjp$-differentiated wave equations \eqref{wave eq q1}. The control of these two terms is identically the same as in the proof of Lemma \ref{lem wave H3} and so we no longer repeat the details. That is, the wave-type estimates immediately lead to
			\begin{align}\label{remainder wave}
				\|\nab q\|_0^2+\|\eps\nab D_t q\|_0^2+\|\eps\nab(\fjp q)\|_0^2\lesssim P(E(0))+\int_0^t P(E(\tau))\dtau.
			\end{align}

			\paragraph*{The source terms of tangentially-differentiated wave equations.} In \eqref{reduce nabq 2}, \eqref{reduce Dtq 2} and \eqref{reduce FPq 2}, we shall seek for the control of
			\[
			\|\wg^\eps\|_1^2,\q \|\eps D_t \wg^\eps\|_0^2,\q \|\eps\flp \wg^\eps\|_0^2.
			\]
			In view of \eqref{wave source 2}, we have
			\begin{align*}
				\|\wg^\eps\|_1^2\lesssim &~\|\nab S\cdot \nab q\|_1^2 +  \left( \|F_j\|_1^2\|\eps\fjp S\|_2^2+\|\eps(\p u)(\p u)\|_1^2+\|\eps(\p F_j)(\p F_j)\|_1^2 +  \|\eps\nab q\cdot\nab q\|_1^2 \right)\\
				&+\|\eps ( \fjp S)( \fjp S) \|_1^2+\|(\eps \fjp q)(\eps \fjp S)\|_1^2  \\
				&+ \eps^2 \left(\|\eps \fjp q\|_1^4+\|\eps D_t q\|_1^4\right).
			\end{align*}Here we omit the terms that $D_t$ or $\fjp$ falls on the coefficients $a,b,\rho$ or $\p_S a, \p_q a, \p_S b, \p_q b$. Since $D_t S=0$ and $\fjp$ is a spatial derivative, we know such terms never lead to loss of $\eps$ weights or derivative loss.
			Invoking Corollary \ref{cor product}, we have
			\[
			\|(\nab U)(\nab V)\|_1^2\lesssim\delta\| U\|_3^2 + \|U\|_1^2\|V\|_2^{16}\lesssim \delta E(t) + \|U\|_1^2\|V\|_2^{16},
			\]where $U, V$ can be any of $u, F_j, q, S$. The $\delta$-terms will be absorbed when we finalize the estimate of $E(t)$, while all the other terms are of lower order. 
			
			For the term $\|\nab S\cdot \nab q\|_1^2$, the above argument gives us 
			\begin{align*}
				\|\nab S\cdot \nab q\|_1^2\lesssim \delta E(t) +\|q\|_1^2\| S \|_2^{16}.
			\end{align*}Since $\|q\|_1^2$ has been controlled in the $L^2$ estimate in Proposition \ref{L2 energy} and the wave-type estimate in Lemma \ref{lem wave L2}, and $\p_t S=-u\cdot\nab S$, we know $\|S(t)\|_2^{16}\lesssim \|S_0\|_2^{16}+\int_0^t \|u\|_2^{16}\|\nab S\|_2^{16}\dtau\leq P(E(0))+\int_0^t P(E(\tau))\dtau$, which then leads to the uniform-in-$\eps$ estimate
			\begin{align}
				\|\nab S\cdot \nab q\|_1^2\lesssim \delta E(t) + P(E(0))+\int_0^t P(E(\tau))\dtau.
			\end{align}
			For the other terms appearing in the estimate of $\|\wg^\eps\|_1^2$, there is at least one term involving $\eps$ weight, and so we get by mimicking the proof of ``Prop. \ref{lem_Esti_div_1} $\Rightarrow$ Prop. \ref{lem_Esti_div}" to get
			\begin{align*}
				\|(\nab U)(\eps\nab V)\|_1^2\lesssim&~\delta E(t) + \|U\|_1^2\|\eps V\|_2^8\lesssim \delta E(t) + \|U\|_1^2\left(\|\eps V(0)\|_2^8+\int_0^t\|\eps\p_t V(\tau,\cdot)\|_2^8\dtau\right)\\
				=&~\delta E(t) + \eps^6\|\eps U\|_1^2\|V(0)\|_2^8 + \|U(t)\|_1^2\int_0^t\|\eps\p_t V(\tau,\cdot)\|_2^8\dtau\\
				\lesssim&~\delta E(t) + \eps^6\|V(0)\|_2^8\left(\|\eps U(0)\|_2^2+\int_0^t\|\eps\p_t U(\tau,\cdot)\|_2^2\dtau\right) + \|U(t)\|_1^2\int_0^t\|\eps\p_t V(\tau,\cdot)\|_2^8\dtau\\
				\lesssim&~\delta E(t) + P(E(0))+(1+E(t))\int_0^t P(E(\tau))\dtau.
			\end{align*}
			Thus, we obtain the uniform-in-$\eps$ control for the source term $\|\rho\wg^\eps\|_1^2$ by
			\begin{align}\label{remainder source 1}
				\forall\delta\in(0,1),~~~\|\wg^\eps\|_1^2\lesssim \delta E(t) + P(E(0))+P(E(t))\int_0^t P(E(\tau))\dtau.
			\end{align}
			
			Next, we seek for the control of $\|\eps D_t (\rho\wg^\eps)\|_0^2$ and $\|\eps\fjp(\rho\wg^\eps)\|_0^2$. Let us recall their concrete forms
			\begin{align*}
				\eps D_t \wg^\eps= &~\eps b \nab S\cdot \nab D_t q+2\eps^2 a\nab q\cdot \nab D_t q  -\eps^2 \rho(\p_i D_t u_j \, \p_j u_i + \eps^2  \p_i D_t F_{kj} \, \p_j F_{ki}) \\
				&-\eps^3 \rho\left(\p_S a + \p_q b\right)\left(\fjp D_t q\right)\left(\fjp S\right) \\
				& - 2\eps^4\rho \p_q a \left(\fjp q\,\fjp D_t q - D_t q\,D_t^2 q\right)+\cdots 
			\end{align*}
			and
			\begin{align*}
				\eps\flp\wg^\eps=&~\eps b \flp (\nab S\cdot \nab q)+\eps^2  \flp |\nab q |^2 -\eps^2 \rho\flp (\p_i u_j \, \p_j u_i +  \p_i F_{kj} \, \p_j F_{ki}) \notag\\
				&-\sum_{j=1}^3\eps^3  \rho(\p_S a + \p_q b)(\fjp(\flp q))(\fjp S)  \notag \\
				&-\sum_{j=1}^3\eps^3  \rho(\p_S a + \p_q b)(\fjp(\flp S))(\fjp q)\notag\\
				&-\sum_{j=1}^3 2\eps^4\rho \p_q a (\fjp q\,(\fjp(\flp q))  - D_t q\,\flp D_t q)+\cdots \notag
			\end{align*}
			where the omitted terms are generated by either of the following two ways:
			\begin{itemize}
				\item $D_t$ or $\fjp$ falls on the coefficients $a,b,\rho$ or $\p_S a, \p_q a, \p_S b, \p_q b$. Since $D_t S=0$ and $\fjp$ is a spatial derivative, we know such terms never lead to loss of $\eps$ weights or derivative loss.
				\item Commute $D_t$ with $\nab$. Note also that $[D_t,\fjp]=0$ and $[D_t,\nab](\cdot)=(\nab u)\tilde{\cdot}\p(\cdot)$ does not lead to loss of $\eps$ weights or derivative loss.
			\end{itemize}
			Here we only show the control of $\|\eps D_t \wg^\eps\|_0^2$ and the same argument applies to $\|\eps\flp\wg^\eps\|_0^2$ by replacing $D_t$ by $\flp$. In fact, we only need to notice that the terms in $\eps D_t \wg^\eps$ have the form
			\[
			\nab S\cdot(\eps\nab D_t q )\text{ or }\eps^j(\eps \nab U)(\eps \nab D_t V),~j\geq 0.
			\] We just follow the same strategy as in the control of $\|\rho\wg^\eps\|_1^2$ to get
			\[
			\|(\eps \nab U)(\eps\nab D_t V)\|_0^2\lesssim\delta\|\eps D_t V\|_2^2 + \|\eps D_t V\|_0^2\|\eps U\|_2^8\leq \delta E(t) + P(E(0))+P(E(t))\int_0^t P(E(\tau))\dtau,
			\]and
			\[
			\|\nab S\cdot(\eps\nab D_t q )\|_0^2\leq \delta \|\eps D_t q\|_2^2 + P(E(0))+\int_0^t P(E(\tau))\dtau.
			\] Thus, we conclude that
			\begin{align}\label{remainder source 2}
				\forall\delta>0,~~\|\eps D_t \wg^\eps\|_0^2+\|\eps\flp\wg^\eps\|_0^2\lesssim \delta E(t) + P(E(0))+P(E(t))\int_0^t P(E(\tau))\dtau.
			\end{align}
			
			\paragraph*{The commutators arising from the elliptic estimates.} In \eqref{reduce Dtq 1} and \eqref{reduce FPq 1}, we need to control $\|\eps[\lap, D_t]q\|_0^2$ and $\|\eps[\lap,\flp]q\|_0^2$. Lemma \ref{commutators} shows that these two terms have the form $\eps \lap X\cdot\nab q + 2 \eps\sum\limits_{i,k=1}^3(\p_i X_k)\,\p_i\p_k q$ with $X=u$ or $X=\fl$. Thus, by using Corollary \ref{cor product}, we have
			\begin{align*}
				\|\eps \lap X\cdot\nab q\|_0^2+2\eps\|(\p X)(\p^2 q)\|_0^2\lesssim \delta(\|X\|_3^2+\|q\|_3^2)+\|X\|_2^2\|\eps\nab q\|_1^4+\|q\|_2^2\|\eps\nab X\|_1^4.
			\end{align*}Then use $\|X\|_2^2\lesssim\|X\|_1\|X\|_3\lesssim\delta\|X\|_3^2+\|X\|_1^2$ and Young's inequality
			\begin{align}
				&\|\eps \lap X\cdot\nab q\|_0^2+2\eps\|(\p X)(\p^2 q)\|_0^2\lesssim \delta(\|X\|_3^2+\|q\|_3^2)+\|X\|_1^2\|\eps\nab q\|_1^8+\|q\|_1^2\|\eps\nab X\|_1^8.
			\end{align} The $\eps$-term can be easily bounded
			\[
			\|\eps\nab X\|_1^8+\|\eps\nab q\|_1^8\lesssim \eps^8\left(\|\nab X(0)\|_1^8+\|\nab q(0)\|_1^8\right)+\int_0^t \|\eps \nab \p_t X(\tau,\cdot)\|_1^8+\|\eps \nab \p_t q(\tau,\cdot)\|_1^8\dtau.
			\]Then we get
			\begin{align*}
				\|X\|_1^2\|\eps\nab q\|_8^2\lesssim&~ \eps^6\|\eps X\|_1^2\left(\|\nab X(0)\|_1^8+\|\nab q(0)\|_1^8\right) + \|X\|_1^2 \int_0^t \|\eps \nab \p_t X(\tau,\cdot)\|_1^8+\|\eps \nab \p_t q(\tau,\cdot)\|_1^8\dtau\\
				\lesssim&~\eps^6(E(0))^4\left(\|\eps X(0,\cdot)\|_1^2 + \int_0^t\|\eps\p_t X(\tau,\cdot)\|_1^2\dtau \right) +(1+E(t))\int_0^tP(E(\tau))\dtau\\
				\lesssim&~P(E(0))+E(t)\int_0^tP(E(\tau))\dtau,
			\end{align*}and similarly
			\[
			\|q\|_1^2\|\eps\nab X\|_8^2\lesssim P(E(0))+E(t)\int_0^tP(E(\tau))\dtau.
			\]
			Therefore, we get the uniform-in-$\eps$ estimates for the commutators involving $\lap$:
			\begin{align}
				\|\eps[\lap, D_t]q\|_0^2+\sum_{l=1}^3\|\eps[\lap,\flp]q\|_0^2\lesssim \delta E(t)+ P(E(0))+P(E(t))\int_0^tP(E(\tau))\dtau
			\end{align}without using the smallness of $\eps$. Here $\delta\in(0,1)$ is a positive constant.
			
			We also need to control $|\eps[\p_3,D_t]q|_{\frac12}^2$, $|\eps[\p_3,\flp]q|_{\frac12}^2$, $\|\eps^2[D_t^2,\nab]q\|_0^2$ and $\|\eps^2[D_t\flp,\nab]q\|_0^2$. Invoking their concrete forms in Lemma \ref{commutators}, we can find that the control of these commutators can be done in the same way as in the control of $\wg^\eps$ and $\eps D_t\wg^\eps$. In fact, the commutators involving the boundary norms can be controlled by
			\[
			|\eps[\p_3,D_t]q|_{\frac12}^2\leq \|\eps[\p_3,D_t]q\|_1^2\leq \|\eps(\p_3 u_k)(\p_k q)\|_1^2
			\]which still has the form $\|\eps(\nab U)(\nab V)\|_1^2$. Similarly, the leading-order terms in $\|\eps^2[D_t^2,\nab]q\|_0^2$ are $\|(\eps \p u)(\eps\p D_t q)\|_0^2$ and $\|(\eps \p q)(\eps\p D_t u)\|_0^2$ which still has the form $\|(\eps \p U)(\eps\p D_tV)\|_0^2$. So, we conclude the commutator estimates by the following inequality without repeating the same steps:
			\begin{align*}
				&\|\eps[\lap, D_t]q\|_0^2+|\eps[\p_3,D_t]q|_{\frac12}^2 + \|\eps^2[D_t^2,\nab] q\|_0^2\\
				+&\sum_{l=1}^3\|\eps[\lap,\flp]q\|_0^2+|\eps[\p_3,\flp]q|_{\frac12}^2+ \|\eps^2[D_t\flp,\nab] q\|_0^2\\
				\lesssim&~ \delta E(t)+ P(E(0))+P(E(t))\int_0^t P(E(\tau))\dtau.
			\end{align*}
		\end{proof}
		
		\subsubsection{Estimates of divergence and pressure}
		Now, we can conclude that Proposition \ref{lem_Esti_div_1} holds. In fact, the control of quantities in Proposition \ref{lem_Esti_div_1} are reduced to the control of the terms in \eqref{reduce div} by repeatedly using the concrete form of the wave equation \eqref{wave eq q2} and the div-curl inequality in Lemma \ref{lem divcurl}, as shown in \eqref{reduce nabq 1}-\eqref{reduce div 3}. Then we establish the control of these norms of $q$ via the wave-type estimates for tangentially-differentiated wave equation \eqref{wave eq q1} (not \eqref{wave eq q2}), which is presented in Lemma \ref{lem wave L2}, Lemma \ref{lem wave H3}. The remainders generated in \eqref{reduce nabq 1}-\eqref{reduce FPq 2} are all controlled in Lemma \ref{lem wave remainders}. Thus, we obtain that
		\[
		\|(\eps D_t)^k\nabla q\|_{2-k}^2 + \|(\eps D_t)^k\nabla \cdot u\|_{2-k}^2  + \sum_{j=1}^3 \|(\eps D_t)^k\nabla \cdot F_j\|_{2-k}^2  \leq \delta E(t) + P(E(0))+ P(E(t))\int_0^t P(E(\tau))\dtau
		\]as desired in Proposition \ref{lem_Esti_div_1} which then immediately leads to Proposition \ref{lem_Esti_div}.
		
		\subsection{Estimates of full time derivatives}
		Now, it remains for us to establish $L^2$-energy estimate for the full-time derivatives of \eqref{CElasto3}. We have:
		\begin{prop}\label{lem full time}
			Under the assumptions of Theorem \ref{main thm, ill data}, we have
			\begin{equation}
				\sum_{k=0}^{3}\ddt \left(\|(\eps\p_t)^k(q,u)\|_0^2  +\sum_{j=1}^{3}\|(\eps\p_t)^k F_j\|_0^2  \right) \le P(E(t)).
			\end{equation}
		\end{prop}
		\begin{proof}
			For $k=0,\cdots, 3$, taking $\p_t^k$ of the first three equations of \eqref{CElasto3} gives
			\begin{equation}\label{equ_p_t^k V}
				\begin{aligned}
					& a D_t \p_t^k q + \eps^{-1}\nabla\cdot \p_t^k u= \mathcal{C}_q,\\
					& \rho D_t \p_t^k u + \eps^{-1} \nabla \p_t^k q= \rho\sum_{j=1}^{3} ( F_j + \BF_j ) \cdot\nabla F_j+\mathcal{C}_u,\\
					& \rho D_t \p_t^k F_j= \rho(F_j + \BF_j) \cdot\nabla \p_t^k u + \mathcal{C}_{F_j}.
				\end{aligned}
			\end{equation}
			Here, the commutators $(\mathcal{C}_q, \mathcal{C}_u, \mathcal{C}_{F_j})$ are given by
			\begin{equation}
				\begin{aligned}
					\mathcal{C}_q=&-[a D_t, \p_t^k] q,\\
					\mathcal{C}_u=&-[\rho D_t, \p_t^k]u + \sum_{j=1}^{3} [\rho(F_j+\BF_j)\cdot\nab, \p_t^k] F_j,\\
					\mathcal{C}_{F_j}=&-[\rho D_t,\p_t^k] F_j + [\rho(F_j+\BF_j)\cdot\nabla ,\p_t^k]u.
				\end{aligned}
			\end{equation}
			Multiplying \eqref{equ_p_t^k V} by $\eps^{2k}\p_t^k(q,u,F_j)$, integrating over $\Omega$, and integrating by parts give
			\begin{align}
				&\ddt \int_{\Om} a|(\eps\p_t)^k q |^2 + \rho|(\eps\p_t)^k u|^2 +\sum_{j=1}^{3} \rho |(\eps\p_t)^k F_j|^2 \dx \notag\\
				=&\int_{\Om} ( \p_t a + \nabla\cdot( a u ) )|(\eps\p_t)^k q|^2 \dx-\eps^{2k-1} \int_{\Om}\nabla\cdot( \p_t^k q \p_t^k u ) \dx \notag\\
				& + \sum_{j=1}^{3} \eps^{2k}\int_{\Om} \rho( F_j+\BF_j)\cdot\nabla \p_t^k F_j \cdot \p_t^k u  + \rho( F_j +\BF_j )\cdot\nabla\p_t^k u \cdot \p_t^k F_j \dx\notag\\
				&+\eps^{2k}\int_{\Om} \p_t^k q \cdot \mathcal{C}_q + \p_t^k u\cdot\p_t^k u + \sum_{j=1}^{3} \p_t^k F_j \cdot\mathcal{C}_{F_j} \dx.\label{EnergyEstimate_p_t^kV}
			\end{align}
			Since $a$ is a smooth function of $(\eps q, S)$, we get from the Sobolev inequality that
			\begin{equation*}
				\|\p_t a+ \nabla\cdot ( a u )\|_{\infty} \le P(\|(q,u,S)\|_3) \le P(E(t)).
			\end{equation*}
			As a result, the first term on the right-hand side of \eqref{EnergyEstimate_p_t^kV} can be bounded by
			\begin{equation}\label{Inter1}
				\int_{\Om} ( \p_t a + \nabla\cdot( a u ) )|(\eps\p_t)^k q|^2 \dx \le \|\p_t a+ \nabla\cdot ( a u )\|_{\infty} \|(\eps\p_t)^k q\|_0^2 \le P(E(t)).
			\end{equation}
			By using the boundary condition \eqref{bdry cond} and Stokes formula, we obtain
			\begin{equation}\label{Inter2}
				-\eps^{2k-1} \int_{\Om}\nabla\cdot( \p_t^k q \p_t^k u ) \dx=0.
			\end{equation}
			By using the boundary condition \eqref{bdry cond} and the fourth equation of \eqref{CElasto3}, we get
			\begin{align}
				& \sum_{j=1}^{3}\int_{\Om} \rho( F_j+\BF_j)\cdot\nabla \p_t^k F_j \cdot \p_t^k u  + \rho( F_j +\BF_j )\cdot\p_t^k u \cdot \p_t^k F_j \dx \notag\\
				=&\sum_{j=1}^{3}\int_{\Om} -\left( \rho( F_j+\BF_j)\cdot\nabla \p_t^k u +  \nabla\cdot(\rho( F_j+\BF_j)) \p_t^k u  \right) \cdot \p_t^k F_j  + \rho( F_j +\BF_j )\cdot\nabla\p_t^k u \cdot \p_t^k F_j \dx \notag\\
				=&~0.\label{Inter3}
			\end{align}
			Finally, we control the fourth term on the right-hand side of \eqref{EnergyEstimate_p_t^kV}. Invoking the concrete forms of the commutators in Lemma \ref{commutators}, we know that the highest-order terms in $[D_t,\p_t^k]$ contains at most $k$ copies of time derivatives and the number of all derivatives does not exceed $k$. Besides, when $\p_t$ falls on $a(\eps q,S)$ or $\rho(\eps q,S)$, there is no loss of $\eps$ weight as we have $\p_t S=-u\cdot\nab S$. Thus, there is no loss of weights of Mach number or the loss of derivatives. A straightforward product estimate then gives us
			\begin{equation}
				\|\eps^k[h\p_t,\p_t^k] f\|_0+ \| \eps^k[h u\cdot\nabla,\p_t^k] f\|_0+ \|\eps^k[h(F_j+\BF_j)\cdot\nabla,\p_t^k]f\|_0\le P( \|(q,u,S,f)\|_{3,\eps} ),
			\end{equation}where $h=h(h\eps q,S)$ is an arbitrary smooth function in its arguments.
			Based on the concrete form of the commutators $(\mathcal{C}_q, \mathcal{C}_u, \mathcal{C}_{F_j})$, we get
			\begin{equation}\label{Inter4}
				\eps^{2k}\int_{\Om} \p_t^k q \cdot \mathcal{C}_q + \p_t^k u\cdot\p_t^k u + \sum_{j=1}^{3} \p_t^k F_j \cdot\mathcal{C}_{F_j} \dx \le P(E(t)).
			\end{equation}
			Substituting \eqref{Inter1}, \eqref{Inter2}, \eqref{Inter3} and \eqref{Inter4} into \eqref{EnergyEstimate_p_t^kV} gives
			\begin{equation}
				\ddt \left(\|(\eps\p_t)^k(q,u)\|_0^2  +\sum_{j=1}^{3}\|(\eps\p_t)^k F_j\|_0^2  \right) \le P(E(t)).
			\end{equation}
			The proof is completed.
		\end{proof}

		\subsection{Closing the uniform estimates}

		Now, we can prove the following lemma by using the results we proved earlier. Recall that we define
		\begin{align*}
			E(t)=&\left\|(q,u,S)\right\|_{3,\eps}^2+\sum_{j=1}^3\left\|(F_j,(F_j+\BF_j)\cdot\nabla S)\right\|_{3,\eps}^2,\\
			E_1(t)=&\sum_{k=0}^{3}\|(\eps\p_t)^k(q,u,F_j)\|^2_0  + \|S\|_{3,\eps}^2 +\sum_{j=1}^3 \|(F_j+\BF_j)\cdot\nabla S\|_{3,\eps}^2+\|\nabla\times (\rho_0 u)\|_{2,\eps}^2 +\sum_{j=1}^3 \|\nabla\times (\rho_0 F_j)\|_{2,\eps}^2,\\
			\label{E2_energy}	E_2(t)=&~\|\nabla q\|_{2,\eps}^2 + \|\nabla\cdot u\|_{2,\eps}^2 +\sum_{j=1}^3 \|\nabla\cdot F_j\|_{2,\eps}^2 +\|\nabla\times u\|_{2,\eps}^2 +\sum_{j=1}^3 \|\nabla\times F_j\|_{2,\eps}^2,
		\end{align*}
		\begin{lem}\label{lem_Esti_E1E2}
			Under the assumptions of Theorem \ref{main thm, ill data}, we have
			\begin{align}
				&\ddt E_1(t) \le P(E(t)),\\
				\forall\delta\in(0,1)~~~	&E_2(t)\le P(E_1(t)) +\delta E(t) + P(E(0))+ P(E(t))\int_0^t P(E(\tau))\dtau ,
			\end{align}
		\end{lem}
		\begin{proof}
			The first inequality follows from Lemma \ref{lem_EnergyEsti_S}, Lemma \ref{lem_EnergyEsti_curl} and Proposition \ref{lem full time}. For the control of $E_2(t)$, Corollary \ref{lem_Esti_curl} and Proposition \ref{lem_Esti_div}  give us
			\begin{align}\label{reduce E2 1}
				E_2(t) \le& \left( 1+\sum_{k=0}^2\sum_{l=0}^{k}\|(\eps\p_t)^l u\|_{2-k}^2 \right)P(E_1(t))+ \delta E(t) + P(E(0))+ P(E(t))\int_0^t P(E(\tau))\dtau,
			\end{align}where the term $\sum\limits_{k=0}^2\sum\limits_{l=0}^{k}\|(\eps\p_t)^l u\|_{2-k}^2$ originates from the curl estimate in Corollary \ref{lem_Esti_curl}. When $k=2$, it has already been controlled in Proposition \ref{lem full time}. When $k=0,1$, we notice that such terms are lower order and thus we can continue to apply the div-curl analysis to them as in Corollary \ref{lem_Esti_curl} and Proposition \ref{lem_Esti_div} such that all the spatial derivatives are reduced to time derivatives. That is, we finally reach the following inequality
			\begin{align}\label{reduce E2 2}
				E_2(t) \le& \left( 1+\sum_{l=0}^{2}\|(\eps\p_t)^l u\|_{0}^2 \right)P(E_1(t))+ \delta E(t) + P(E(0))+ P(E(t))\int_0^t P(E(\tau))\dtau\notag\\
				\leq &~ P(E_1(t))+ \delta E(t) + P(E(0))+ P(E(t))\int_0^t P(E(\tau))\dtau
			\end{align}as desired, because $\sum\limits_{l=0}^{2}\|(\eps\p_t)^l u\|_{0}^2$ is already a part of $E_1(t)$.
		\end{proof}

		Finally, combining the estimates presented in Lemma \ref{lem_Esti_E1E2}, we obtain the Gr\"onwall-type inequality
		\begin{align*}
			\forall \delta\in(0,1),~~E(t)\lesssim &~E_1(t)+E_2(t)\leq \delta E(t) + P(E(0))+ P(E(t))\int_0^t P(E(\tau))\dtau.
		\end{align*}By selecting $\delta>0$ suitably small (independent of $\eps$) such that $\delta E(t)$ can be absorbed by the left side, we obtain the following uniform-in-$\eps$ estimate
		\[
		E(t)\lesssim P(E(0))+ P(E(t))\int_0^t P(E(\tau))\dtau.
		\]Thus, there exists a time $T>0$ that only depends on $E(0)$ and does not depend on $\eps$, such that
		\[
		\sup_{t\in[0,T]}E(t) \le P(E(0)),
		\]which completes the proof of Theorem \ref{main thm, ill data}.

		\section{Incompressible limit}\label{sect limit}
		In this section, we aim to prove the strong convergence for the solution of \eqref{CElasto3} and the limit is the incompressible inhomogeneous elastodynamic system \eqref{IElasto}. By the uniform estimate \eqref{energy estimate} and the equation
		\begin{align}\label{H1p_tcurl}
			D_t( \nabla\times(\rho_0 u) ) -\sum_{j=1}^{3}\nabla\times\left( \rho_0 (F_j+\BF_j)\cdot\nabla F_j \right) =[D_t,\nabla\times](\rho_0 u) + \nabla g \times\nabla q,
		\end{align}
		we get
		\begin{equation}
			\sup_{t\in[0,T]}\|(\p_t\nabla\times(\rho_0 u))(t)\|_1 \le P(E(0)).
		\end{equation}
		The uniform estimates \eqref{energy estimate} and \eqref{H1p_tcurl} imply that, up to a subsequence, we have the weak-* convergence for all variables and the strong convergence for the deformation tensor $F_j~(j=1,2,3)$, the entropy $S$ and the vorticity of the fluid velocity.
		\begin{align}
			(q,u,F_j,S) \to (q^0,u^0,F_j^0,S^0) &\text{ weakly-* in } L^{\infty}([0,T];H^3(\Omega)),\\
			(F_j,S)\to (F^0_j,S^0) &\text{ strongly in } C([0,T];H^{3-\delta}_{\mathrm{loc}}(\Omega)),\label{StrongLim_S_B}\\
			\nabla\times(\rho_0(S)u) \to \nabla\times(\rho_0(S^0)u^0) &\text{ strongly in } C([0,T];H^{2-\delta}_{\mathrm{loc}}(\Omega)),\label{StrongLim_curlrho0u}
		\end{align}
		with $\rho_0(S)=\rho(0,S)$. Similarly, we use $a_0(S)$ denote $a(0,S)$. In view of the div-curl inequality in Lemma \ref{lem divcurl}, we shall prove the strong convergence of $q$, the divergence $\nab\cdot u$ and the $L^2(\Om)$ norm of $u-u^0$.

		\subsection{Strong convergence of pressure and divergence}
		Before proving the strong convergence of $q$ and $\nab\cdot u$, we need to find out the expected value for their limits. We prove that $q^0=0$ and $\nabla\cdot u^0=0$. The first and second equation of \eqref{CElasto3} are written to 
		\begin{equation}
			E(\eps q, S) D_t U + \eps^{-1} L U= J,
		\end{equation}
		with
		\begin{equation*}
			E(\eps q, S)=\begin{pmatrix}
				a(\eps q, S) & 0 \\
				0 & \rho(\eps q, S) I_3
			\end{pmatrix},\quad L=\begin{pmatrix}
				0 &\nabla\cdot\\
				\nabla & 0
			\end{pmatrix},\quad U=\begin{pmatrix}
				q \\
				u
			\end{pmatrix},\quad J=-\begin{pmatrix}
				0\\
				\rho\sum\limits_{j=1}^{3}(F_j+\BF_j )\cdot\nabla F_j
			\end{pmatrix}.
		\end{equation*}
		First notice that
		\begin{equation}
			\eps E(\eps q, S) \p_t U + L U = -\eps E(\eps q, S) u\cdot\nabla U +\eps J.
		\end{equation}
		Using the uniform bounds and $E(\eps q,S)-E_0(S)=O(\eps)$, we obtain
		\begin{equation}\label{EuqU}
			\eps E_0(S)\p_t U + L U= \eps f,
		\end{equation} 
		where $E_0(S)= E(0,S)$ and $\{f\}_{\eps>0}$ is a bounded family in $C([0,T];H^2(\Omega))$. Passing to the weak limit shows that $\nabla q^0=0$ and $\nabla\cdot u^0$=0. Since $q^0\in L^\infty([0,T];H^3(\Omega))$ and $\Omega=\mathbb{R}^3_-$, we infer $q^0=0$.

		\begin{prop}\label{StrongLim_q_divu}
			Under the assumptions of Theorem \ref{main thm, limit}, we have
			\begin{align}
				q\to&~ 0 \text{ strongly in } L^2([0,T];H^{3-\delta}_{\mathrm{loc}}(\Omega)),\\
				\nabla\cdot u \to &~ 0 \text{ strongly in } L^2([0,T];H^{2-\delta}_{\mathrm{loc}}(\Omega)).\label{StrongLim_divu}
			\end{align}
		\end{prop}
		
		\begin{proof}
			This is a slight variant in Alazard \cite[Prop. 3.1]{Alazard2005limit}, so we will only outline the proof and skip some technical details that are identical to \cite[Prop. 3.1]{Alazard2005limit}. 
			\paragraph*{Step 1: Wave-packet transform. }To get the strong compactness in time, the idea is to construct the defect measures associated to the sequence $q$ and $\nabla\cdot u$ and then prove they vanish. The first step is doing the wave-packet transform to convert the time variable to frequency variable. More precisely, one first extends the
			functions to $t\in \mathbb{R}$ by
			\begin{equation}
				\tilde{U}=\begin{pmatrix}
					\tilde{q}\\
					\tilde{u}
				\end{pmatrix}=\chi_{\eps} U =\begin{pmatrix}
					\chi_{\eps} q\\
					\chi_{\eps}u
				\end{pmatrix},
			\end{equation}
			where $\chi_{\eps}\in C_0^\infty((0,T))$ be a family of functions such that $\chi_{\eps}(t)=1$ for $t\in [\eps^{1/2},T-\eps^{1/2}]$ and $\|\eps\p_t \chi_{\eps}\|_{\infty} \le 2\eps^{1/2}$, and choose extensions $\tilde{S}$ of $S$, supported in $t\in [-1,T+1]$, uniformly bounded in $C(\mathbb{R};H^3(\Omega))$, and converging to $\tilde{S}^0$ in $C(\mathbb{R};H^{3-\delta}_{\mathrm{loc}}(\Omega))$. According to \eqref{EuqU}, $\tilde{U}$ satisfies
			\begin{equation}\label{EuqTildeU}
				\eps E_0(\tilde{S})\p_t \tilde{U} + L \tilde{U}= \eps \tilde{f},
			\end{equation} 
			where $\{\tilde{f}\}_{\eps>0}$ is a bounded family in $C(\mathbb{R};H^2(\Omega))$. By using the wave packet transform:
			\begin{equation}
				W^{\eps} v(t,\tau,x)= (2\pi^3)^{-1/4}\eps^{-3/4} \int_{\mathbb{R}} e^{(\mathrm{i} (t-s)\tau -(t-s)^2 )/\eps } v(s,x) \mathrm{d} s,
			\end{equation}
			where $v\in C^1(\mathbb{R}\times\bar{\Om}) \cap L^2(\mathbb{R}\times \Omega)$,  $W^{\eps} v\in C^1(\mathbb{R}^2_{t,\tau}\times\bar{\Om}) \cap L^2(\mathbb{R}^2_{t,\tau}\times \Omega)$ and $W^{\eps}$ extends as an isometry from $L^2(\mathbb{R}\times\Omega)$ to $L^2(\mathbb{R}^2_{t,\tau}\times\Omega)$, \eqref{EuqTildeU} can be written in $\mathbb{R}^2_{t,\tau}\times \Omega$ as
			\begin{equation}\label{EquWTildeU}
				\mathrm{i}\tau E_0(\tilde{S})( W^{\eps}\tilde{U}) + L ( W^{\eps}\tilde{U} ) = \boldsymbol{G}^{\eps},
			\end{equation}
			where 
			\begin{align*}
				\boldsymbol{G}^{\eps}=&~\eps W^{\eps} \tilde{f} + [ E_0(\tilde{S}),W^{\eps} ](\eps\p_t) \tilde{U} + E_0(\tilde{S}) ( \mathrm{i}\tau W^{\eps} \tilde{U}- W^{\eps} (\eps\p_t \tilde{U} )  ) \\
				:=&~( G^{\eps}_1, \boldsymbol{G}_2^{\eps} ) \in L^2(\mathbb{R}^2_{t,\tau};H^1(\Omega)  ) \times L^2(\mathbb{R}^2_{t,\tau};(H^1(\Omega))^3 ).
			\end{align*}
			Following \cite[Lemma 3.3]{Alazard2005limit}, one can show that
			\begin{equation}\label{FepsTo0}
				\boldsymbol{G}^{\varepsilon}\to 0 \text{ in } L^2(\mathbb{R}^2_{t,\tau};H^1(\Omega)) \text{ as } \varepsilon\to 0.
			\end{equation}
			\paragraph*{Step 2: Decomposition of the pressure.}
			Since the wave-packet transform is an isometry between $L^2(\R_{t,\tau}^2\times\Om)$ and $L^2(R\times\Om)$, it suffices to prove the strong convergence for $W^\eps \tilde{q}$.  From \eqref{EquWTildeU}, we obtain
			\begin{equation}
				P^{\eps}(t,\tau,\nabla) (W^{\eps}\tilde{q})= -\mathrm{i}\tau G_1^{\eps} + \nabla\cdot( \rho_0^{-1}(\tilde{S}) \boldsymbol{G}_2^{\eps}  ).
			\end{equation}
			where
			\begin{align}
				P^{\eps}(t,\tau,\nabla) (\cdot):=& a_0(\tilde{S})\tau^2 (\cdot) + \nabla\cdot ( \rho_0^{-1}(\tilde{S}) \nabla(\cdot) ),\\
				P^{0}(t,\tau,\nabla) (\cdot):=& a_0(\tilde{S}^0)\tau^2 (\cdot) + \nabla\cdot ( \rho_0^{-1}(\tilde{S}^0) \nabla(\cdot) ).
			\end{align}
			Since we are now considering a boundary-value problem, we shall decompose $W^\eps \tilde{q}$ into its interior part and the boundary part. Following Alazard \cite[Section 3.2]{Alazard2005limit}, we have
			\begin{align}\label{decompose Wq}
				W^\eps \tilde{q} = (1-\lap_N)^{-1}\Theta + \mathfrak{N} (\boldsymbol{G}_2^\eps\cdot N),
			\end{align}where $\Theta:=(1-\lap)(W^\eps\tilde{q})$ and $(1-\Delta_N)^{-1}$ is defined by
			\begin{equation*}
				f=(1-\Delta_N)^{-1}g \text{ if and only if } (1-\Delta)f=g \text{ in }\Omega,
				\text{ and } \p_N f=0 \text{ on }\Sigma;	
			\end{equation*}and $\mathfrak{N}$ is defined by
			\begin{equation*}
				h=\mathfrak{N}(g) \text{ if and only if } (1-\Delta)h=0 \text{ in }\Omega,
				\text{ and } \p_N h=g \text{ on }\Sigma.
			\end{equation*} It should be noted that $(1-\lap_N)^{-1}$ is a bounded linear operator from $L^2(\Om)$ to $H^2(\Om)$ and $\mathfrak{N}$ is a bounded linear operator from $H^{\frac12}(\Sigma)$ to $H^2(\Om)$.
			
			\paragraph*{Step 3: Strong convergence.}
			The strong convergence of $\boldsymbol{G}_2^\eps$ in \eqref{FepsTo0} implies that
			\begin{align}
				\mathfrak{N} (\boldsymbol{G}_2^\eps\cdot N)\to 0~~\text{ in }L^2(\R_{t,\tau}^2; H^2(\Om)),
			\end{align}and also we have for any $\varphi\in C_0(\R^2)$ and any compact operator $K$ on $L^2(\Om)$ that
			\begin{align}
				\varphi KP^\eps(t,\tau,\nab)\mathfrak{N} (\boldsymbol{G}_2^\eps\cdot N)\to 0~~\text{ in }L^2(\R_{t,\tau}^2\times\Om).
			\end{align}
			To prove the strong convergence of $\Theta$, which is now only a uniformly bounded family in $L^2(\R^2\times\Om)$, we need the following two lemmas.
			\begin{lem}[M\'etivier-Schochet {\cite[Lemma 4.3]{Metivier2001limit}}]\label{lem MS}
				For all uniformly bounded family $\{\Theta^\eps\}\subset L^2(\R^{2+d})$, there is a subsequence such that there exists a finite non-negative Borel measure $\mu$ on $\R^2$ and $M\in L^1(\R^2,\mathcal{L}_+,\mu)$ such that for all $\Phi\in C_0(\R^2;\mathcal{K})$,
				\[
				\int_{\R^2}\left( \Phi \Theta^\eps, \Theta^\eps \right)_{L^2}\dt\dtau \xrightarrow{\eps \to 0}\int_{\R^2} \text{Tr}(\Phi(t,\tau)M(t,\tau))\mu(\dt,\dtau).
				\]Here $\mathcal{K}$ ($\mathcal{L}_+$, resp.) denotes the set of compact operators (non-negative self-adjoint trace class operators, resp.) on $L^2(\Om)$.
			\end{lem}
			
			\begin{lem}[M\'etivier-Schochet {\cite[Lemma 5.1]{Metivier2001limit}}]\label{lem ker}
				The operator $P^0(t,\tau,\nab)(1-\lap_N)^{-1}=0$ is a 1-1 mapping for any $(t,\tau)\in\R^2$, that is,
				\begin{align}
					\ker_{L^2(\Omega)} ( P^0(t,\tau,\nabla)(1-\Delta_N)^{-1}  ) =0, \quad \forall (t,\tau)\in \mathbb{R}^2\label{KerP0}
				\end{align}
			\end{lem}
			\begin{rmk}
				The entropy decay condition \eqref{EntropyDecay} is necessary in the proof of \cite[Lemma 5.1]{Metivier2001limit}. This also explains why we require the domain $\Om$ to be unbounded, for example, the half space.
			\end{rmk}
			Let $M(t,\tau)$ be the trace-class operator and $\mu$ be the microlocal defect measure obtained in Lemma \ref{lem MS} by inserting $\Theta^\eps$ defined in \eqref{decompose Wq}. Then we can prove
			\begin{cor}\label{cor ker}
				\begin{align}
					M(t,\tau)=0 \quad \mu\text{-a.e.} \label{P0Ma.e.}
				\end{align}
			\end{cor}
			\begin{proof}[Proof of Corollary \ref{cor ker}]
				Following  \cite[Prop. 3.4]{Alazard2005limit}, we can prove that
				\[
				\varphi K \left(P^0(t,\tau,\nab)(1-\lap_N)^{-1}\Theta\right)\to 0~~\text{ in }L^2(\R_{t,\tau}^2\times\Om).
				\]Thus, we set $\Phi(t,\tau):=\varphi(t,\tau)KP^0(t,\tau,\nab)(1-\lap_N)^{-1}$ in Lemma \ref{lem MS} with $\varphi\in C_0(\R^2)$ and $K\in\mathcal{K}$ to get
				\[
				\int_{\R^2}\left(\varphi K P^0(t,\tau,\nab)(1-\lap_N)^{-1}\Theta, \Theta \right)_{L^2(\Om)}\dt\dtau\xrightarrow{\eps \to 0} 0,
				\]which together with Lemma \ref{lem ker} forces 
				\[
				\int_{\R^2} \text{Tr}(\varphi K P^0(t,\tau,\nab)(1-\lap_N)^{-1}M(t,\tau))\mu(\dt,\dtau)=0.
				\]Since $\varphi$ and $K$ are arbitrary, we get $P^0(t,\tau,\nabla)(1-\Delta_N)^{-1} M(t,\tau)=0$ for $\mu$-a.e. $(t,\tau)\in\R^2$. Since $M(t,\tau)$ is a bounded operator on $L^2(\Om)$, we know Lemma \ref{lem ker} then leads to $M(t,\tau)=0$ for $\mu$-a.e. $(t,\tau)\in\R^2$.
			\end{proof}
			
			We now set $\Phi(t,\tau)=\varphi(t,\tau)K^*K$ for $\varphi\in C_0(\R^2)$ and $K\in\mathcal{K}$ in Lemma \ref{lem MS} to get
			\begin{equation}
				\varphi K\Theta^{\eps} \to 0 \text{ in } L^2(\mathbb{R}^2_{t,\tau}\times\Omega)
			\end{equation}holds for any $K\in \mathcal{K}$ and $\varphi\in C_0(\R^2)$.
			Following the arguments in \cite[(3.23)-(3.24)]{Alazard2005limit}, we prove this convergence holds for $\varphi(t,\tau)=1$, i.e. for any $K\in \mathcal{K}$,
			\begin{equation}\label{ThetaCovergence1}
				K \Theta^{\eps} \to 0 \text{ in } L^2(\mathbb{R}^2_{t,\tau}\times\Omega).
			\end{equation}
			Recall that, by the definition of $\Theta=(1-\Delta) (W^{\eps} q)$, $W^{\eps}$ is an isometry from $L^2(\mathbb{R}_t\times\Omega)$ to $L^2(\mathbb{R}^2_{t,\tau}\times \Omega)$, and $W^{\eps}$ commutes with $K(1-\Delta)$. So \eqref{ThetaCovergence1} implies that for any $K\in \mathcal{K}$,
			\begin{equation}\label{qConvergence1}
				K(1-\Delta)\tilde{q}\to 0 \text{ in } L^2(\mathbb{R}\times\Omega).
			\end{equation}
			Given that $\tilde{q}$ is bounded in $L^2(\mathbb{R};H^3(\Omega))$, the convergence \eqref{qConvergence1} implies
			\begin{equation}
				\tilde{q}\to 0 \text{ in } L^2(\mathbb{R};H^{3-\delta}_{\mathrm{loc}}(\Omega)).
			\end{equation}
			Since the limit is $0$, the convergence holds without passing a subsequence. We end up with
			\begin{equation}
				q\to 0 \text{ in } L^2([0,T];H^{3-\delta}_{\mathrm{loc}}(\Omega)).
			\end{equation}
			Arguments similar to those above show that, from \eqref{EquWTildeU}, we have that
				\begin{equation}
					\eps\p_t q\to 0 \text{ in } L^2([0,T];H^{2-\delta}_{\mathrm{loc}}(\Omega)),
				\end{equation}
			which deduces that
			\begin{equation}
				\nabla \cdot u = -a D_t q\to 0 \text{ in } L^2([0,T];H^{2-\delta}_{\mathrm{loc}}(\Omega)).
			\end{equation}
		\end{proof}
		
		\subsection{The limit process to the incompressible inhomogeneous elastodynamic system}
		
		We continue our proof of Theorem \ref{main thm, limit}. It now remains to prove the strong convergence of $u-u^0$. Recall that $\mathcal{P}$ be the projection onto $H_{\sigma}$ and $\mathcal{Q}= I_3 - \mathcal{P}$, where $H_\sigma=\{u\in L^2(\Omega): \int_{\Om} u\cdot\nabla\phi,\ \forall \phi\in H^1(\Omega)\}$ and $G_{\sigma}=\{\nabla\psi: \psi\in H^1(\Omega)\}$ give the orthogonal decomposition $L^2(\Omega)=H_{\sigma} \oplus G_{\sigma}$.
		
		From the strong convergence results \eqref{StrongLim_curlrho0u} and \eqref{StrongLim_divu}, we know that
		\begin{align}
			\mathcal{P}( \rho_0(S)u )\to \mathcal{P}( \rho_0(S^0)u^0 ) &\text{ in } L^2([0,T];H^{3-\delta}_{\mathrm{loc}}(\Omega)),\\
			\mathcal{Q}u\to \mathcal{Q}u^0=0 &\text{ in } L^2([0,T];H^{3-\delta}_{\mathrm{loc}}(\Omega)).\label{StrongLim_Qu}
		\end{align}
		The previous two properties yields further that:
		\begin{align}
			\mathcal{P}( \rho_0(S)\mathcal{P}u ) \to \mathcal{P}(\rho_0(S^0)\mathcal{P}u^0) &\text{ in } L^2([0,T];H^{3-\delta}_{\mathrm{loc}}(\Omega)),\\
			\mathcal{P}( \rho_0(S)\mathcal{Q}u ) \to \mathcal{P}(\rho_0(S^0)\mathcal{Q}u^0)=0 &\text{ in } L^2([0,T];H^{3-\delta}_{\mathrm{loc}}(\Omega)),
		\end{align}
		which, combined with the fact $S\to S^0$ in $C([0,T];H^{3-\delta}_{\mathrm{loc}}(\Omega))$, imply that:
		\begin{align*}
			\mathcal{P}( \rho_0(S^0)\mathcal{P}(u-u^0) )=& ~\mathcal{P}( \rho_0(S^0)[ (u-u^0)- \mathcal{Q}u ]  ) \\
			=&~\mathcal{P}\left(  \rho_0(S)(u-u^0) + ( \rho_0(S^0)-\rho_0(S) )(u-u^0) - \rho_0(S) \mathcal{Q} u + (\rho_0(S^0)-\rho_0(S))\mathcal{Q} u  \right)\\
			\to &~0 \text{ in } L^2([0,T];H^{3-\delta}_{\mathrm{loc}}(\Omega)).
		\end{align*}
		Now, we recall that $\rho_0(S^0)$ is strictly positive in $[0,T]\times \Omega$ and $\mathcal{P}u\to \mathcal{P}u^0= u^0$ in $L^2([0,T];H^{3-\delta}_{\mathrm{loc}}(\Omega))$, which together with \eqref{StrongLim_Qu} implies that
		\begin{equation}\label{StrongLim_u}
			u\to u^0 \text{ in } L^2([0,T];H^{3-\delta}_{\mathrm{loc}}(\Omega)).
		\end{equation}
		By \eqref{StrongLim_S_B} and \eqref{StrongLim_u}, we obtain
		\begin{align*}
			\rho(\eps q,S)\to \rho_0(S^0) &\text{ in } C([0,T];H^{3-\delta}_{\mathrm{loc}}(\Omega)),\\
			\nabla u\to \nabla u^0,~~ \nabla F_j\to \nabla F_j^0&\text{ in } L^2([0,T];H^{2-\delta}_{\mathrm{loc}}(\Omega)).
		\end{align*}
		Passing to the limit in the equations for $S$ and $F_j$, we see that the limits $S^0$ and $F_j^0$ satisfy
		\begin{equation*}
			(\partial_t +u^0\cdot\nabla) S^0=0,\quad (\partial_t + u^0\cdot\nabla)F_j^0 = (F_j^0+\BF_j) \cdot\nabla u^0,\quad \nabla\cdot (\rho_0(S^0)( F_j^0+\BF_j ))=0
		\end{equation*}
		in the sense of distributions. Since $\rho(\eps q, S)- \rho_0(S)=O(\eps)$, we have
		\begin{align*}
			\rho(\eps q, S) D_t u= (\rho(\eps q, S)- \rho_0(S))D_t u + \p_t ( \rho_0(S)u ) + (u\cdot \nabla)( \rho_0(S) u )\to ~\rho_0(S^0) ( (\p_t +u^0\cdot\nabla)u^0  )
		\end{align*}
		in the sense of distributions. Applying the operator $\mathcal{P}$ to the momentum equation $\rho D_t u + \eps^{-1}\nabla q =\rho\sum\limits_{j=1}^{3}(F_j+\BF_j)\cdot\nabla F_j$ and then taking to the limit, we conclude that
		\begin{equation*}
			\mathcal{P}\left[  \rho_0(S^0)( (\p_t +u^0\cdot\nabla)u^0 -\rho_0(S^0)\sum_{j=1}^{3}(F_j^0+\BF_j)\cdot\nabla F_j^0 \right]=0.
		\end{equation*}
		Therefore, $(u^0,F_j^0,S^0)\in C([0,T];H^3(\Om))$ solves the incompressible inhomogeneous elastodynamic equations together with a transport equation
		\begin{equation}
			\begin{cases}
				\varrho(\p_t u^0 + u^0\cdot\nab u^0) + \nab \pi =\varrho\sum\limits_{j=1}^{3}(F_j^0+\BF_j)\cdot\nabla F_j^0&~~~ \text{in}~[0,T]\times \Omega,\\
				\p_t F_j^0=(F_j^0+\BF_j)\cdot\nabla u^0 &~~~ \text{in}~[0,T]\times \Omega,\\
				\nab\cdot u^0=0,\quad \nab\cdot (\varrho (F_j^0+\BF_j))=0&~~~ \text{in}~[0,T]\times \Omega,\\
				\p_t S^0+u^0\cdot\nab S^0=0&~~~ \text{in}~[0,T]\times \Omega,\\
				u_3^0=F_{3j}^0=0&~~~\text{on}~[0,T]\times\Sigma,
			\end{cases}
		\end{equation}
		for a suitable fluid pressure function $\pi$ satisfying $\nab\pi\in C([0,T];H^2(\Om))$. Here $\varrho$ satisfies $\p_t\varrho+u^0\cdot\nab\varrho=0,$ with initial data $\varrho_0:=\rho(0,S_0^0)$. Employing the arguments in \cite[Theorem 1.5]{Metivier2001limit}, we find that 
		\begin{equation*}
			(u^0,F_j^0,S^0)|_{t=0}=(w_0, F_{j,0}^0,S^0_0),
		\end{equation*}
		where $w_0 \in H^3(\Omega)$ is determined by
		\begin{equation*}
			w_{03}|_{\Sigma}=0,\quad \nabla\cdot w_0=0,\quad \nabla\times( \rho_0(S_0^0)w_0 )=\nabla\times( \rho_0(S_0^0)u_0^0 ).
		\end{equation*}
		Moreover, the uniqueness of the limit function implies that the convergence holds as $\eps\to 0$ without restricting to a subsequence. Theorem \ref{main thm, limit} is then proven.
		
		\paragraph*{Acknowledgment.} The authors would like to express their gratitude to the reviewers for careful reading and helpful suggestions which led to an improvement of the original manuscript. Jiawei Wang was supported by the NSFC (Grants No. 12131007 and 12601401), the Basic Science Center Program, China (Grant No. 12288201) of the NSFC, the China Postdoctoral Science Foundation (Grant No. 2025T180846), and the Postdoctoral Fellowship Program of CPSF (Grant No. GZC20252032). Junyan Zhang was supported in part by the NSFC (Grants No. 12601416).

		\begin{appendix}
		\section{Preliminary lemmas}\label{sect lemma}
		In this section, we record several lemmas that are repeatedly used throughout this manuscript. The first lemma records the div-curl decomposition for a vector field.
		\begin{lem}[Hodge-type elliptic estimates]\label{lem divcurl}
			For any sufficiently smooth vector field $X$ and any real number $s\ge 1$, we have
			\begin{equation}\label{elliptic estimates}
				\|X\|_{s}^2 \lesssim \|X\|_0^2 + \|\nabla\cdot X\|_{s-1}^2 + \|\nabla\times X\|_{s-1}^2+ |X\cdot N|_{s-\frac12}^2,
			\end{equation}
		\end{lem}
		
		The next lemma records Kato-Ponce type multiplicative Sobolev inequalities.
		\begin{lem}[{\cite{KatoPonce1988}}, Kato-Ponce type inequalities] \label{KatoPonce} For any $s\geq 0$, we have 
			\begin{equation}\label{product}
					\|fg\|_{H^s}\lesssim \|f\|_{W^{s,p_1}}\|g\|_{L^{p_2}}+\|f\|_{L^{q_1}}\|g\|_{W^{s,q_2}},\qquad \|fg\|_{\dot{H}^s}\lesssim \|f\|_{\dot{W}^{s,p_1}}\|g\|_{L^{p_2}}+\|f\|_{L^{q_1}}\|g\|_{\dot{W}^{s,q_2}},
			\end{equation}with $1/2=1/p_1+1/p_2=1/q_1+1/q_2$ and $2\leq p_1,q_2<\infty$.
		\end{lem}
		In particular, we repeatedly use the following multiplicative Sobolev inequality throughout the manuscript.
		\begin{cor}\label{cor product}
			Assume $f,g\in H^2(\Om)$. Then for any constant $\delta\in(0,1)$, we have
			\[
			\|fg\|_1^2\lesssim \delta \|g\|_2^2 +  \|f\|_0^{16}\|g\|_{1}^2,\q\q \|fg\|_0^2\lesssim \delta \|g\|_1^2 +  \|f\|_1^4\|g\|_{0}^2.
			\]
		\end{cor}
		\begin{proof}
			Invoking the Kato-Ponce inequality in Lemma \ref{KatoPonce} with $s=1$, $(p_1,p_2)=(2,\infty),~(q_1,q_2)=(6,3)$, Sobolev embeddings $H^1(\Om)\hookrightarrow L^6(\Om),~H^{1.5}(\Om)\hookrightarrow W^{1,3}(\Om), H^{1.5+\alpha}(\Om)\hookrightarrow L^\infty(\Om)$ (for $\alpha>0$) and the interpolation inequality, we get 
			\[
			\|fg\|_1^2\lesssim \|f\|_1^2\|g\|_{L^\infty}^2+\|f\|_{L^6}^2\|g\|_{W^{1,3}}^2\lesssim  \|f\|_1^2\|g\|_{1.5+\alpha}^2\lesssim \|f\|_1^2\|g\|_1^{1-2\alpha}\|g\|_2^{1+2\alpha}\lesssim \delta \|g\|_2^2 +  \|f\|_1^{\frac{4}{1-2\alpha}}\|g\|_{1}^2.
			\]Taking $\alpha=\frac14$ for instance and we obtain $\|fg\|_{1}^2\lesssim  \|g\|_2^2 +  \|f\|_1^{8}\|g\|_{1}^2.$ Using interpolation and Young's inequality again, we get
			\[
			\|f\|_1^8\|g\|_{1}^2\leq \|f\|_1^8\|g\|_0\|g\|_2\lesssim \delta\|g\|_2^2 + \|f\|_1^{16}\|g\|_0^2
			\]as desired. The second inequality is proved in the same way:
			\[
			\|fg\|_0^2\leq \|f\|_{L^6}^2\|g\|_{L^3}^2\lesssim\|f\|_1^2\|g\|_{\frac12}^2\lesssim\|f\|_1^2\|g\|_1\|g\|_0\lesssim \delta\|g\|_1^2+\|f\|_1^4\|g\|_0^2.
			\]
		\end{proof}
		
		The next lemma records the concrete forms of several commutators repeatedly used in this manuscript.
		\begin{lem}[{\cite[Section 4]{Luo2018CWW}}]\label{commutators}
			We have $[\p_a, D_t]=(\p_a u)\tilde{\cdot}\p$ for $a=t,1,2,3$, where the symmetric dot product $(\p u)\tilde{\cdot} \p$ is defined component-wisely by $(\p_a u)\tilde{\cdot}\p=\p_a u_l \p_l$. For $k\geq 2$, we have
			\begin{align}
				[\p, D_t^k]=&~ (\p D_t^{k-1} u)\tilde{\cdot} \p + k (\p u)\tilde{\cdot}( \p D_t^{k-1} ) \notag\\
				&~+\sum_{l_1+l_2=k-1\atop l_1,l_2 >0 } c_{l_1,l_2} ( \p D_t^{l_1}u )\tilde{\cdot}(\p D_t^{l_2}) + \sum_{l_1+\cdots+l_n=k-n+1 \atop n\ge 3} d_{l_1,\cdots,l_n}(\p D_t^{l_1}u)\cdots(\p D_t^{l_{n-1}}u)(\p D_t^{l_n})
			\end{align} for some $c_{l_1,l_2}, d_{l_1,l\cdots,l_n}\in\Z$ and
			\begin{align}
				[D_t, \p^k]=\sum_{j=0}^{k-1} c_{j,k} (\p^{1+j}u)\tilde{\cdot}\p^{k-j}
			\end{align}for some $c_{j,k}\in\Z$.
			
			For any (sufficiently smooth) vector field $X$, we have
			\begin{align}
				[\lap,X\cdot\nab ](\cdot)=\lap X\cdot\nab(\cdot) +2\sum_{i,j=1}^3 (\p_i X_j)\p_i\p_j(\cdot)
			\end{align} and 
			\begin{align}
				[\lap, D_t](\cdot)=\lap u\cdot\nab(\cdot) +2\sum_{i,j=1}^3 (\p_i u_j)\p_i\p_j(\cdot).
			\end{align}
		\end{lem}
		
		\end{appendix}


\begin{thebibliography}{99}
			\addcontentsline{toc}{section}{References}
			\setlength{\itemsep}{0.5ex}
			\begin{spacing}{1}
				\bibitem{Alazard2005limit}
				Alazard, T.
				\newblock {\em Incompressible limit of the nonisentropic Euler equations with the
					solid wall boundary conditions.}
				\newblock Adv. Differ. Equ., 10(1):19--44, 2005.
				
				\bibitem{Asano1987limit}
				Asano, K.
				\newblock {\em On the incompressible limit of the compressible Euler equation.}
				\newblock Japan J. Appl. Math., 4(3):455--488, 1987.
				
				\bibitem{CHWWY2}
				Chen, R. M., Hu, J., Wang, D. 
				\newblock {\em Linear stability of compressible vortex sheets in 2D elastodynamics: variable coefficients.}
				\newblock Math. Ann., 376(3): 863–912, 2020.
				
				\bibitem{Cheng2012} 
				Cheng, B. 
				\newblock{\em Singular Limits and Convergence Rates of Compressible Euler and Rotating Shallow Water Equations.}
				\newblock SIAM J. Math. Anal., 44(2): 1050-1076. 
				
				\bibitem{CSdivcurl}
				Cheng, C.-H. A., Shkoller, S.
				\newblock {\em Solvability and Regularity for an Elliptic System Prescribing the Curl, Divergence, and Partial Trace of a Vector Field on Sobolev-Class Domains.}
				\newblock J. Math. Fluid Mech., 19(3): 375-422, 2017.
				
				\bibitem{CuiHu2022ARMA}
				Cui, X., Hu, X.
				\newblock Incompressible limit of three dimensional compressible viscoelastic
				systems with vanishing shear viscosity.
				\newblock {\em Arch. Ration. Mech. Anal.}, 245(2):753--807, 2022.
				
				\bibitem{CuiHu2023Vanishing}
				Cui, X., Hu, X.
				\newblock Vanishing shear viscosity limit and incompressible limit of two
				dimensional compressible dissipative elastodynamics.
				\newblock {\em Commun. Math. Phys.}, 402(3):3253--3336, 2023.
				
				
				
				\bibitem{Dafermos10}
				Dafermos, C. M.
				\newblock {\em Hyperbolic Conservation Laws in Continuum Physics, 3rd edition,} Grundlehren Math. Wiis., Vol. 325, Springer-Verlag, 2010.
				
				\bibitem{Ebin1982limit}
				Ebin, D. G.
				\newblock {\em Motion of slightly compressible fluids in a bounded domain. I.}
				\newblock Commun. Pure Appl. Math., 35(4):451-485, 1982.
				
				
				\bibitem{FangAndZi2014incompressible}
				Fang, D., Zi, R.
				\newblock Incompressible limit of {Oldroyd-B} fluids in the whole space.
				\newblock {\em J. Differential Equations}, 256(7):2559--2602, 2014.
				
				
				\bibitem{GuillopeSalloum2010DCDS}
				Guillop\'e, C., Salloum, Z, Talhouk, R.
				\newblock Regular flows of weakly compressible viscoelastic fluids and the
				incompressible limit.
				\newblock {\em Discrete Contin. Dyn. Syst. Ser. B}, 14(3):1001--1028, 2010.
				
				\bibitem{HuOuWangYang2023AdvNA}
				Hu, X., Ou, Y., Wang, D., Yang, L.
				\newblock Incompressible limit for compressible viscoelastic flows with large
				velocity.
				\newblock {\em Adv. Nonlinear Anal.}, 12(1):Paper No. 20220324, 20, 2023.
				
				
				
				\bibitem{Iguchi1997limit}
				Iguchi, T.
				\newblock {\em The incompressible limit and the initial layer of the compressible
					Euler equation in $\mathbb{R}^{n}_+$.}
				\newblock Math. Methods Appl. Sci., 20(11):945-958, 1997.
				
				
				\bibitem{Isozaki1987limit}
				Isozaki, H.
				\newblock{\em Singular limits for the compressible Euler equations in an exterior domain.}
				\newblock J. Reine Angew. Math., 381:1-36, 1987.
				
				\bibitem{JuWang2023}
				Ju, Q., Wang, J.
				\newblock {\em Low mach number limit of nonisentropic inviscid {Hookean}
					elastodynamics.}
				\newblock Math. Methods Appl. Sci., 46(8):9508--9525, 2023.
				
				\bibitem{JuWangXu2022}
				Ju, Q., Wang, J., Xu, X.
				\newblock {\em Low Mach number limit of inviscid Hookean elastodynamics.}
				\newblock Nonlinear Anal. Real World Appl., 68:103683, 2022.
				
				\bibitem{KatoPonce1988}
				Kato, T., Ponce, G.
				\newblock {\em Commutator estimates and the Euler and Navier-Stokes equations.}
				\newblock Commun. Pure Appl. Math., 41(7): 891-907, 1988.
				
				\bibitem{Klainerman1981limit}
				Klainerman, S., Majda, A.
				\newblock {\em Singular limits of quasilinear hyperbolic systems with large parameters and the incompressible limit of compressible fluids.}
				\newblock Commun. Pure Appl. Math., 34(4):481--524, 1981.
				
				\bibitem{Klainerman1982limit}
				Klainerman, S., Majda, A.
				\newblock {\em Compressible and incompressible fluids.}
				\newblock Commun. Pure Appl. Math., 35(5):629--651, 1982.
				
				\bibitem{LeiAndZhou2005global}
				Lei, Z., Zhou, Y.
				\newblock Global existence of classical solutions for the two-dimensional
				{Oldroyd} model via the incompressible limit.
				\newblock {\em SIAM J. Math. Anal.}, 37(3):797--814, 2005.
				
				\bibitem{Lei2006global}
				Lei, Z.
				\newblock Global existence of classical solutions for some {Oldroyd-B} model
				via the incompressible limit.
				\newblock {\em Chinese Ann. Math. Ser. B}, 27(5):565--580, 2006.
				
				
				
				\bibitem{LiuXu2021}
				Liu, G., Xu, X.
				\newblock {\em Incompressible limit of the Hookean elastodynamics in a bounded domain.}
				\newblock  Z. Angew. Math. Phys., 72:81, 1-14, 2021.
				
				\bibitem{Luo2018CWW}
				Luo, C. 
				\newblock {\em On the Motion of a Compressible Gravity Water Wave with Vorticity.}
				\newblock Ann. PDE, 4(2): 2506-2576, 2018.
				
				\bibitem{LuoZhang2022CWWST}
				Luo, C., Zhang, J. 
				\newblock {\em Compressible Gravity-Capillary Water Waves: Local Well-Posedness, Incompressible and Zero-Surface-Tension Limits.}
				\newblock arXiv:2211.03600, preprint, 2022.			
				
				
				\bibitem{Metivier2001limit}
				M\'etivier, G., Schochet, S.
				\newblock {\em The incompressible limit of the non-isentropic {E}uler equations.}
				\newblock Arch. Rational Mech. Anal., 158(1):61-90, 2001.
				
				\bibitem{Rauch1985}
				Rauch, J.
				\newblock {\em Symmetric Positive Systems with Boundary Characteristic of Constant Multiplicity.}
				\newblock Trans. Amer. Math. Soc., 291(1), 167-187, 1985.
				
				
				\bibitem{RenAndOu2015strong}
				Ren, D., Ou, Y.
				\newblock Strong solutions to an {Oldroyd-B} model with slip boundary
				conditions via incompressible limit.
				\newblock {\em Math. Methods Appl. Sci.}, 38(2):330--348, 2015.
				
				
				\bibitem{Sideris1991}
				Sideris, T.
				\newblock {\em  The lifespan of smooth solutions to the three-dimensional compressible Euler equations and the incompressible limit.}
				\newblock Indiana Univ. Math. J., 40(2): 535-550, 1991.
				
				\bibitem{Sideris2005}
				Sideris, T., Thomases, B.
				\newblock {\em  Global existence for three-dimensional incompressible isotropic elastodynamics via the incompressible limit.}
				\newblock Commun. Pure Appl. Math., 58(6): 750-788, 2005.
				
				\bibitem{Schochet1985elasticity}
				Schochet., S.
				\newblock {\em The incompressible limit in nonlinear elasticity.}
				\newblock Commun. Math. Phys., 102(2):207--215, 1985.
				
				\bibitem{Schochet1986limit}
				Schochet, S.
				\newblock {\em The compressible Euler equations in a bounded domain: Existence of solutions and the incompressible limit.}
				\newblock Commun. Math. Phys., 104(1):49--75, 1986.
				
				\bibitem{Schochet1994limit}
				Schochet, S.
				\newblock {\em Fast Singular Limits of Hyperbolic PDEs.}
				\newblock J. Differ. Equ., 114(2):476-512, 1994.
				
				\bibitem{Secchi2000}
				Secchi, P.
				\newblock {\em On the Singular Incompressible Limit of Inviscid Compressible Fluids.}
				\newblock  J. Math. Fluid Mech., 2(2), 107-125, 2000.
				
				\bibitem{Secchi2002}
				Secchi, P.
				\newblock {\em On slightly compressible ideal flow in the halfplane.}
				\newblock Arch. Rational Mech. Anal., 161(3): 231-255, 2002. 
				
				\bibitem{Trakhinin2016elastic}
				Trakhinin, Y.
				\newblock {\em Well-posedness of the free boundary problem in compressible elastodynamics.}
				\newblock  J. Differ. Eq. 264(3): 1661-1715, 2018.
				
				\bibitem{VSphysical}
				Truesdell, C., Toupin, R.
				\newblock {\em The classical field theories}, with an appendix on tensor fields by J.L. Ericksen, in: S. Fl\"ugge(Ed.), 
				\newblock Handbuch der Physik, Bd. III/1, Springer, Berlin, 1960, pp. 226-793, appendix, pp. 794-858.
				
				\bibitem{Wang2022elasto}
				Wang, J.
				\newblock {\em Incompressible limit of nonisentropic Hookean elastodynamics.}
				\newblock J. Math. Phys., 63(6):061506, 2022.
				
				\bibitem{WangZhang2025JDE}
				Wang, J., Zhang, J.
				\newblock Incompressible limit of compressible ideal {MHD} flows inside a
				perfectly conducting wall.
				\newblock {\em J. Differential Equations}, 425:846--894, 2025.
				
				\bibitem{Ukai1986limit}
				Ukai, S.
				\newblock {\em The incompressible limit and the initial layer of the compressible Euler equation.}
				\newblock J. Math. Kyoto Univ., 26(2):323-331, 1986.
				
				\bibitem{ZhangZhong2025JLMS}
				Zhang, J., Zhong, M.
				\newblock The combined singular limits of compressible {O}ldroyd-{B} model at
				low {M}ach and {W}eissenberg numbers.
				\newblock {\em J. Lond. Math. Soc. (2)}, 112(1):Paper No. e70243, 50, 2025.
				
				
				\bibitem{Zhang2021elasto}
				Zhang, J. 
				\newblock {\em Local Well-posedness and Incompressible Limit of the Free-Boundary Problem in Compressible Elastodynamics.}
				\newblock Arch. Rational Mech. Anal., 244(3), 599-697, 2022.
				
				
				
				
				
				
				
				
				
				
				
				
				
				
				
				
				
				
				
				
				
				
				
				
				
			\end{spacing}
		\end{thebibliography}



	\end{document}